\documentclass[reqno]{amsart}
\usepackage[english]{babel}
\usepackage{graphicx,subfigure}
\usepackage{amsmath,amssymb}
\usepackage{appendix}
\usepackage{color}
\definecolor{oneblue}{rgb}{0,0.0,0.75}

\usepackage{fancyhdr}

\makeatletter
\newcommand{\sech}{\mathop{\operator@font sech}}
\newcommand{\sign}{\mathop{\operator@font sign}}
\makeatother

\newtheorem{theorem}{Theorem}[section]
\newtheorem{proposition}{Proposition}[section]

\newtheorem{remark}{Remark}[section]
\numberwithin{equation}{section}

\begin{document}

\title[]{The numerical solution of 2D Boussinesq/Boussinesq models for internal waves with spectral methods}


\author[A. Duran]{Angel Duran}
\address{ Applied Mathematics Department,  University of
Valladolid, 47011 Valladolid, Spain}
\email{angeldm@uva.es}




\begin{abstract}
The numerical approximation of some Boussinesq systems in two spatial dimensions is here considered. The differential systems under study are proposed as asymptotic models for the propagation of waves along the interface of two layers of fluids with different densities and subjected to a Boussinesq physical regime in each layer. Well-posedness of the periodic initial-value problem (ivp) of the systems is first analized. Then, a discretization in space based on the spectral Fourier-Galerkin method is introduced and error estimates for the semidiscrete approximation are derived. Using an efficient time integrator, some numerical experiments to illustrate the performance of the discretization are presented.
\end{abstract}

\maketitle

\section{Introduction}
The present paper is concerned with the numerical approximation of the family of PDEs
\begin{eqnarray}
(1-\mu b\Delta)\zeta_{t}+r_{1}\nabla\cdot{\bf v}_{\beta}+\epsilon \lambda\nabla\cdot \left(\zeta{\bf v}_{\beta}\right)+\mu a\nabla\cdot\Delta {\bf v}_{\beta}=0,\nonumber&&\\
(1-\mu d\Delta)({\bf v}_{\beta})_{t}+r_{2}\nabla\zeta+\frac{\epsilon}{2}\lambda\nabla |{\bf v}_{\beta}|^{2}+\mu(1-\gamma)c\Delta\nabla\zeta=0.&&\;\;\label{BB1}
\end{eqnarray}
The system (\ref{BB1}) was introduced in \cite{BLS2008} as a model (in nondimensional formulation) for the wave propagation along the interface of a two-layer system of inviscid, homogeneous fluids under the following conditions. Let $d_{j}, \rho_{j}, j=1,2$, be, respectively, the depth and density of the $j$th layer, with $\rho_{2}>\rho_{1}$. These parameters define the corresponding density and depth ratios
$$\gamma=\frac{\rho_{1}}{\rho_{2}}<1, \delta=\frac{d_{1}}{d_{2}},$$ with
$$\lambda=\frac{\delta^{2}-\gamma}{(\delta+\gamma)^{2}}, r_{1}=\frac{1}{\gamma+\delta}, r_{2}=(1-\gamma),$$ in (\ref{BB1}). A rigid-lid condition is assumed for layer 1 at the top $z=0$, while the bottom of layer 2 (at $z=-(d_{1}+d_{2})$) is assumed to be rigid as well and flat. The deviation of the interface at position $(x,y)$ from the rest and time $t$ is given in (\ref{BB1}) by $\zeta=\zeta(x,y,t)$. The velocity of the wave is represented by some variable $v$, with $v_{\beta }=(I-\beta \Delta )^{-1}v$, being $\Delta$ the Laplace operator and $\beta\geq 0$ a modelling parameter. The system (\ref{BB1}) is derived in \cite{BLS2008} under the so-called Boussinesq/Boussinesq (B/B) regime, meaning that the flow is under the Boussinesq regime in both fluid domains, and is represented by the conditions
\begin{eqnarray*}
 \mu\sim\mu_{2}\sim\epsilon\sim\epsilon_{2}<<1,\label{BB21bb}
\end{eqnarray*}
with $\delta\approx 1$ and where
\begin{equation*}
\epsilon=\frac{a}{d_{1}},\; \mu=\frac{d_{1}}{\lambda^{2}},\; \epsilon_{2}=\frac{a}{d_{2}},\; \mu_{2}=\frac{d_{2}}{\lambda^{2}},\label{param}
\end{equation*}
being $a$ and $\lambda$, respectively,  amplitude and wavelength of the wave; thus $\mu, \epsilon$ measure respectively dispersive and nonli ear effects with respect to layer 1, with similar meaning for $\mu_{2},\epsilon_{2}$ with respect to layer 2.

The formulation of (\ref{BB1}) in \cite{BLS2008} defines four parameters $a,b,c,d$, given by
\begin{eqnarray*}
a&=&\frac{(1-\alpha_{1})(1+\gamma\delta)-3\delta\beta (\delta+\gamma)}{3\delta (\gamma+\delta)^{2}},\quad b=\alpha_{1}\frac{1+\gamma\delta}{3\delta(\gamma+\delta)},\nonumber\\
c&=&\beta\alpha_{2},\quad d=\beta(1-\alpha_{2}),\label{BB1b}
\end{eqnarray*}
where $\alpha_{1}\geq 0, \alpha_{2}\leq 1$ are additional modelling parameters. (They are introduced from a \lq BBM trick\rq\ in order to get some symmetry in the contribution of the third-order derivatives in the equations, which is convenient for well-posedness, \cite{S}.)
For the derivation of B/B and other models for internal waves, for both free-surface and rigid-lid conditions, we refer to \cite{Duchene2021,Lannes2} and references therein. 

As proved in \cite{BLS2008}, the consistency of the corresponding Euler equations for internal waves with (\ref{BB1}) has a precision $O(\epsilon^{3/2})$, uniformly, under certain conditions, with respect to $\epsilon\in [0,1), \mu\in [0,1]$. Linear well-posedness of the initial-value problem (ivp) holds when $a,c\leq 0, b, d\geq 0$ from the dispersion relation
\begin{eqnarray}
\omega^{2}&=&|{\bf k}|^{2}\alpha(k)^{2},\nonumber\\
\alpha(k)^{2}&=&\frac{\left(\frac{1}{\gamma+\delta}-\mu a|{\bf k}|^{2}\right)(1-\gamma)(1-\mu c|{\bf k}|^{2})}{(1+\mu b|{\bf k}|^{2})(1+\mu d|{\bf k}|^{2})}.\label{BB1c}
\end{eqnarray}
On the other hand, nonlinear well-posedness in suitable Sobolev spaces is analyzed in \cite{A}. The study is divided into two groups of systems in (\ref{BB1}): weakly dispersive (for which $b>0$ and $d>0$) and other admissible systems (where $b=0$ or $d=0$), and with $a,c\leq 0$ in both cases. The Sobolev space where the correspnding ivp is well posed depends on the order of the term $\alpha(k)$ in (\ref{BB1c}) and the proofs make use of contraction-mapping arguments and energy methods, respectively; see the discussions in \cite{BonaChS2004,DDS1} for Boussinesq surface and internal wave models in the one-dimensional case.

To the best of our knowledge, the numerical approximation of (\ref{BB1}) is considered for the surface-wave case, corresponding to (\ref{BB1}) in the limit $\gamma=0, \delta=1$, in several papers for one and two dimensions, \cite{DMS2007,DMS3,AD2,ADM1,XRAA}, and for the internal wave propagation in 1D in \cite{XRAA,DDS0}. Focused on the 2D case, in \cite{DMS2007} several $L^{2}$ and $H^{1}$ error estimates of Galerkin-Finite Elements semidiscretizations for the initial-boundary-value-problem (ibvp) with homogeneous Dirichlet boundary conditions of Bona-Smith systems ($b=d>0, a=0, c<0$) for surface waves are derived. The analysis can be adapted to the case of homogeneous Neumann boundary conditions for the surface wave elevation and homogeneous Dirichlet conditions for the velocity variables (called reflective boundary conditions, see also \cite{DMS3}), and some indications on the derivation of $L^{2}$ error estimates for the periodic problem and discretizations with smooth, periodic splines on uniform meshes, are provided.

In this sense, the use of spectral methods to approximate surface- and internal-wave Boussinesq systems seems to be considered only in 1D. The main purpose of the paper is then to study the numerical approximation of the periodic ivp for (\ref{BB1}) with Fourier spectral discretizations in space. This is developed according to the following highlights and structure:
\begin{itemize}
\item The well-posedness study of the ivp of (\ref{BB1}), made in \cite{A}, is first extended to the periodic ivp in section \ref{sec2}.
\item A Fourier-Galerkin spectral semidiscretization in space for the periodic ivp of (\ref{BB1})  is introduced in section \ref{sec3}. Several properties of the resulting semidiscrete system are proved and error estimates for several well-posed B/B systems are derived. They extend the convergence results obtained in \cite{DDS0} for the one-dimensional case. These estimates also cover the approximation of the surface wave case ($\gamma=0, \delta=1$) with Fourier-Galerkin spectral methods. The use of spectral approximations for other types of ibvp's by using Jacobi polynomials, and extending the approach in 1D made in \cite{Duran2023}, is analyzed elsewhere, \cite{Duran2026_2}.
\item We complete the study with the introduction of an efficient full discretization and the illustration of the convergence results with some numerical experiments in section \ref{sec4}.
\end{itemize}
The following notation will be used throughout the paper. For a domain $\Omega\subset \mathbb{R}^{2}$ and $s\in\mathbb{R}$, let $H^{s}=H^{s}(\Omega)$ be the $L^{2}$-based Sobolev space on $\Omega$ of order $s$ (with $L^{2}=H^{0}$) and norm denoted by $||\cdot||_{s}$, being $||\cdot||$ and $(\cdot,\cdot)$ the norm and inner product in $L^{2}$, respectively, and $|\cdot |_{\infty}$ the norm in $L^{\infty}=L^{\infty}(\Omega)$. For $T>0$, $L^{\infty}(0,T,H^{s})$ will denote the normed space of functions $u:[0,T]\rightarrow H^{s}$ with norm
$$||u||_{L^{\infty}(0,T,H^{s})}={\rm ess sup}_{t\in [0,T]}||u(t)||_{s},$$ and ${\rm ess sup}$ as the essential supremum. Similarly, for an integer $k\geq 0$, $C^{k}(0,T,H^{s})$ will denote the space of $k$th-order continuously differentiable functions $u:[0,T]\rightarrow H^{s}$. Finally, the symbol $f\lesssim g$ will denote the existence of some constant $C$ such that $f\leq C g$. In what follows the constant $C$ will be independent of the discretization parameters, but will generally depend on the final time $T$ of approximation and the norm of the exact solution and its spatial derivatives.
\section{Well-posedness of the periodic ivp}
\label{sec2}
Let $T>0$. In this section we consider the periodic ivp (period one, for simplicity) of (\ref{BB1}), written in the form
\begin{eqnarray}
(1-b\Delta)\zeta_{t}+r_{1}\nabla\cdot{\bf v}_{\beta}+\lambda\nabla\cdot \left(\zeta{\bf v}\right)+a\nabla\cdot\Delta {\bf v}=0,\nonumber&&\\
(1-d\Delta)({\bf v})_{t}+r_{2}\nabla\zeta+\frac{\lambda}{2}\nabla |{\bf v}|^{2}+c'\Delta\nabla\zeta=0,&&\;\;\label{BB21}
\end{eqnarray}
for $1$-periodic, real functions $\zeta=\zeta(x,y,t), {\bf v}=(v_{1}(x,y,t),v_{2}(x,y,t))$, $0\leq t\leq T$, $(x,y)\in\Omega=(0,1)^{2}$, $c'=(1-\gamma)c$. The initial conditions
\begin{eqnarray}
\zeta(x,y,0)=\zeta_{0}(x,y),\quad {\bf v}(x,y,0)=v^{0}(x,y)=(v_{1}^{0}(x,y),v_{2}^{0}(x,y)),\label{BB22}
\end{eqnarray}
are $1$-periodic, real functions. Well-posedness of (\ref{BB21}), (\ref{BB22}) will be analyzed here by checking the arguments used in \cite{A} for the ivp to extend the corresponding results. The first case under study concerns the weakly dispersive B/B systems, see also \cite{DMS2007}.
\begin{theorem}
\label{theo21}
Let $s>0$.
\begin{itemize}
\item[(i)] Assume $b,d>0, a,c<0$ or $b,d>0, a=c\leq 0$. Let $(\zeta_{0},v^{0})\in (H^{s})^{3}$. Then there are some $T>0$ and a unique solution $(\zeta,v)\in X_{T}=C(0,T,H^{s})^{3}$ of (\ref{BB21}), (\ref{BB22}).
\item[(ii)] Assume $b,d>0, a=0,c<0$. Let $(\zeta_{0},v^{0})\in H^{s+1}\times (H^{s})^{2}$. Then there are some $T>0$ and a unique solution $(\zeta,v)\in X_{T}=C(0,T,H^{s+1})\times C(0,T,H^{s})^{2}$ of (\ref{BB21}), (\ref{BB22}).
\item[(iii)] Assume $b,d>0, a<0,c=0$. Let $(\zeta_{0},v^{0})\in H^{s-1}\times (H^{s})^{2}$. Then there are some $T>0$ and a unique solution $(\zeta,v)\in X_{T}=C(0,T,H^{s-1})\times C(0,T,H^{s})^{2}$ of (\ref{BB21}), (\ref{BB22}).
\end{itemize}
\end{theorem} 
\begin{proof}
The contraction-mapping arguments used in \cite{A,DMS2007} can be adapted here from the corresponding Fourier representation (with Fourier series instead of Fourier transform). We will follow the same steps and prove, by way of illustration, the case (i). We write (\ref{BB21}) as
\begin{eqnarray}
\zeta_{t}+(I-b\Delta)^{-1}\left(r_{1}\nabla\cdot{\bf v}_{\beta}+\lambda\nabla\cdot \left(\zeta{\bf v}\right)+a\nabla\cdot\Delta {\bf v}\right)=0,\nonumber&&\\
({\bf v})_{t}+(I-d\Delta)^{-1}\left(r_{2}\nabla\zeta+\frac{\lambda}{2}\nabla |{\bf v}|^{2}+c'\Delta\nabla\zeta=0\right).&&\;\;\label{BB21b}
\end{eqnarray}
For ${\bf k}=(k_{x},k_{y})\in\mathbb{Z}^{2}$, let $\widehat{\zeta}({\bf k},t), \widehat{v_{1}}({\bf k},t), \widehat{v_{2}}({\bf k},t)$ be the ${\bf k}$th Fourier coefficient of $\zeta,v_{1},v_{2}$, respectively. The representation of (\ref{BB21b}) in the Fourier space is given by
\begin{eqnarray}
\frac{d}{dt}\begin{pmatrix}\widehat{\zeta}\\\widehat{v_{1}}\\\widehat{v_{2}}\end{pmatrix}({\bf k},t)+i|{\bf k}|\mathcal{A}({\bf k})\begin{pmatrix}\widehat{\zeta}\\\widehat{v_{1}}\\\widehat{v_{2}}\end{pmatrix}({\bf k},t)+iF(\widehat{\zeta},\widehat{v_{1}},\widehat{v_{2}})=0,\label{BB23}
\end{eqnarray}
where $|{\bf k}|=\sqrt{k_{x}^{2}+k_{y}^{2}}$ and
\begin{eqnarray}
\mathcal{A}({\bf k})&=&\begin{pmatrix}0&\frac{k_{x}}{|{\bf k}|}\frac{(r_{1}-a|{\bf k}|^{2})}{1+b|{\bf k}|^{2}}&\frac{k_{y}}{|{\bf k}|}\frac{(r_{1}-a|{\bf k}|^{2})}{1+b|{\bf k}|^{2}}\\
\frac{k_{x}}{|{\bf k}|}\frac{(r_{2}-c'|{\bf k}|^{2})}{1+d|{\bf k}|^{2}}&0&0\\
\frac{k_{y}}{|{\bf k}|}\frac{(r_{2}-c'|{\bf k}|^{2})}{1+d|{\bf k}|^{2}}&0&0\end{pmatrix},\nonumber\\
F&=&\begin{pmatrix}\lambda\frac{k_{x}\widehat{(\zeta v_{1})}({\bf k},t)+k_{y}\widehat{(\zeta v_{2})}({\bf k},t)}{1+b|{\bf k}|^{2}}\\
\frac{\lambda}{2}\frac{k_{x}}{1+d|{\bf k}|^{2}}\widehat{|v|^{2}}({\bf k},t)\\
\frac{\lambda}{2}\frac{k_{y}}{1+d|{\bf k}|^{2}}\widehat{|v|^{2}}({\bf k},t)\end{pmatrix}.\label{BB24}
\end{eqnarray}
The eigenvalues of $\mathcal{A}({\bf k})$ are $\{0,\pm\sigma({\bf k})\}$, where
\begin{eqnarray*}
\sigma({\bf k})=\left(\frac{(r_{1}-a|{\bf k}|^{2})(r_{2}-c'|{\bf k}|^{2})}{(1+b|{\bf k}|^{2})(1+d|{\bf k}|^{2})}\right)^{1/2}=\sqrt{r}\left(\frac{(1-a'|{\bf k}|^{2})(1-c|{\bf k}|^{2})}{(1+b|{\bf k}|^{2})(1+d|{\bf k}|^{2})}\right)^{1/2},
\end{eqnarray*}
with $r=r_{1}r_{2}=\frac{1-\gamma}{\delta+\gamma}>0, a'=(\delta+\gamma)a$. We diagonalize the system (\ref{BB23}), (\ref{BB24}) from
\begin{eqnarray*}
P({\bf k})=\begin{pmatrix}0&\alpha({\bf k})&-\alpha({\bf k})\\-\frac{k_{y}}{|{\bf k}|}&\frac{k_{x}}{|{\bf k}|}&\frac{k_{x}}{|{\bf k}|}\\\frac{k_{x}}{|{\bf k}|}&\frac{k_{y}}{|{\bf k}|}&\frac{k_{y}}{|{\bf k}|}\end{pmatrix},\; 
\alpha({\bf k})=\sqrt{\frac{r_{1}}{r_{2}}}\left(\frac{(1+d|{\bf k}|^{2})(1-a'|{\bf k}|^{2})}{(1+b|{\bf k}|^{2})(1-c|{\bf k}|^{2})}\right)^{1/2}.
\end{eqnarray*}
The change of variables, cf. \cite{DMS2007}
\begin{eqnarray*}
\begin{pmatrix}\widehat{\eta}\\\widehat{w_{1}}\\\widehat{w_{2}}\end{pmatrix}=P^{-1}({\bf k})\begin{pmatrix}\widehat{\zeta}\\\widehat{v_{1}}\\\widehat{v_{2}}\end{pmatrix},
\end{eqnarray*}
that is
\begin{eqnarray}
\widehat{\eta}({\bf k},t)&=&-\frac{k_{y}}{|{\bf k}|}\widehat{v_{1}}({\bf k},t)+\frac{k_{x}}{|{\bf k}|}\widehat{v_{2}}({\bf k},t),\label{BB25}\\
\widehat{w_{1}}({\bf k},t)&=&\frac{1}{2\alpha({\bf k})}\widehat{\zeta}({\bf k},t)+\frac{k_{x}}{2|{\bf k}|}\widehat{v_{1}}({\bf k},t)+\frac{k_{y}}{2|{\bf k}|}\widehat{v_{2}}({\bf k},t),\nonumber\\
\widehat{w_{2}}({\bf k},t)&=&-\frac{1}{2\alpha({\bf k})}\widehat{\zeta}({\bf k},t)+\frac{k_{x}}{2|{\bf k}|}\widehat{v_{1}}({\bf k},t)+\frac{k_{y}}{2|{\bf k}|}\widehat{v_{2}}({\bf k},t),\nonumber
\end{eqnarray}
leads to the system
\begin{eqnarray}
\frac{d}{dt}\begin{pmatrix}\widehat{\eta}\\\widehat{w_{1}}\\\widehat{w_{2}}\end{pmatrix}({\bf k},t)+i|{\bf k}|\begin{pmatrix}0&&\\&\sigma({\bf k})&\\&&-\sigma({\bf k})\end{pmatrix}\begin{pmatrix}\widehat{\eta}\\\widehat{w_{1}}\\\widehat{w_{2}}\end{pmatrix}(k,t)\nonumber&&\\
+iP({\bf k})^{-1}F(\widehat{\eta},\widehat{w_{1}},\widehat{w_{2}})=0.&&\label{BB23b}
\end{eqnarray}
According to (\ref{BB25}), we have the following cases:
\begin{itemize}
\item[(A)] $\alpha({\bf k})$ of order $0$ ($b,d>0, a,c<0$ or $b,d>0, a=c\leq 0$). Then:
\begin{eqnarray*}
(\zeta,v_{1},v_{2})\in (H^{s})^{3}\Rightarrow (\eta,w_{1},w_{2})\in (H^{s})^{3}.
\end{eqnarray*}
\item[(B)] $\alpha({\bf k})$ of order $-1$ ($b,d>0, a=0,c<0$). Then:
\begin{eqnarray*}
(\zeta,v_{1},v_{2})\in H^{s+1}\times (H^{s})^{2}\Rightarrow (\eta,w_{1},w_{2})\in (H^{s})^{3}.
\end{eqnarray*}
\item[(C)] $\alpha({\bf k})$ of order $1$ ($b,d>0, a<0,c=0$). Then:
\begin{eqnarray*}
(\zeta,v_{1},v_{2})\in H^{s-1}\times (H^{s})^{2}\Rightarrow (\eta,w_{1},w_{2})\in (H^{s})^{3}.
\end{eqnarray*}
\item[(D)] $\alpha({\bf k})$ of order $2$ ($b=0,d>0, a<0,c=0$). Then:
\begin{eqnarray*}
(\zeta,v_{1},v_{2})\in H^{s-2}\times (H^{s})^{2}\Rightarrow (\eta,w_{1},w_{2})\in (H^{s})^{3}.
\end{eqnarray*}
\item[(E)] $\alpha({\bf k})$ of order $-2$ ($b>0,d=0, a=0,c<0$). Then:
\begin{eqnarray*}
(\zeta,v_{1},v_{2})\in H^{s+2}\times (H^{s})^{2}\Rightarrow (\eta,w_{1},w_{2})\in (H^{s})^{3}.
\end{eqnarray*}
\end{itemize}
Theorem \ref{theo21} concerns the first three cases. By way of illustration, we consider the proof of (i), writing (\ref{BB23b}) as
\begin{eqnarray*}
\frac{d}{dt}\begin{pmatrix}{\eta}\\{w_{1}}\\{w_{2}}\end{pmatrix}+\mathcal{B}\begin{pmatrix}{\eta}\\{w_{1}}\\{w_{2}}\end{pmatrix}+\mathcal{F}({\eta},{w_{1}},{w_{2}})=0,
\end{eqnarray*}
where $\mathcal{B}$ is the operator with Fourier symbol
$$i|{\bf k}|\begin{pmatrix}0&&\\&\sigma({\bf k})&\\&&-\sigma({\bf k})\end{pmatrix},$$ and
the Fourier representation of $\mathcal{F}$ is
$$iP({\bf k})^{-1}F(\widehat{\eta},\widehat{w_{1}},\widehat{w_{2}}).$$ We apply the periodic version of Lemma 2.2 in \cite{DMS2007} (see also Lemma 2.1 in \cite{A}) to have
\begin{eqnarray*}
||\partial_{x_{k}}(I-r\Delta)^{-1}(fg)||_{s}\leq ||fg||_{s-1}\leq C||f||_{s}||g||_{s},
\end{eqnarray*} 
for $s>0, x_{k}=x,y, r=b>0$ or $d>0$. Then $\mathcal{F}:(H^{s})^{3}\rightarrow (H^{s})^{3}$ is bilinear continuous and contraction-mapping argument leads to well-posedness for $s>0$.
\end{proof}
Well-posedness of other admissible B/B systems (when $b=0$ or $d=0$) are considered in the following result. The classification depends on the order of $\alpha({\bf k})$ and those remaining cases mentioned in the proof of Theorem \ref{theo21}, cf. \cite{A}.
\begin{theorem}
\label{theo22}
\begin{itemize}
\item[(i)] Assume $b=0, d>0, a<0, c=0$. Let $s>1$ and $(\zeta_{0},v^{0})\in H^{s}\times (H^{s+2})^{2}$. Then there are some $T>0$ and a unique solution $(\zeta,v)\in X_{T}=C(0,T,H^{s})\times C(0,T,H^{s+2})^{2}$ of (\ref{BB21}), (\ref{BB22}).
\item[(ii)] Assume $b=0, d>0, a<0, c<0$. Let $s>1$ and $(\zeta_{0},v^{0})\in H^{s}\times (H^{s+1})^{2}$. Then there are some $T>0$ and a unique solution $(\zeta,v)\in X_{T}=C(0,T,H^{s})\times C(0,T,H^{s+1})^{2}$ of (\ref{BB21}), (\ref{BB22}).
\item[(iii)] Assume $b>0, d=0, a<0, c=0$ or $b=0, d>0, a=0, c<0$. Let $s>2$ and $(\zeta_{0},v^{0})\in (H^{s})^{3}$. Then there are some $T>0$ and a unique solution $(\zeta,v)\in X_{T}=C(0,T,H^{s})^{3}$ of (\ref{BB21}), (\ref{BB22}).
\item[(iv)] Assume $b>0, d=0, a<0, c<0$ or $b=0, d>0, a=c=0$. Let $s>2$ and $(\zeta_{0},v^{0})\in H^{s+1}\times (H^{s})^{2}$. Then there are some $T>0$ and a unique solution $(\zeta,v)\in X_{T}=C(0,T,H^{s+1})\times C(0,T,H^{s})^{2}$ of (\ref{BB21}), (\ref{BB22}).
\end{itemize}
\end{theorem}
\begin{proof}
The proof of the corresponding results for the ivp in \cite{A} depends on a priori estimates and energy methods based on tools such as commutator estimates of Kato-Ponce type, the embedding
\begin{eqnarray}
H^{s}\hookrightarrow L^{\infty}, s>1,\label{*1}
\end{eqnarray} 
and classical inequalities, all valid in periodic Sobolev spaces. As in the previous result, we prove the case (i) by way of ilustration. The system (\ref{BB21}) has here the form
\begin{eqnarray}
\zeta_{t}+r_{1}\nabla\cdot{\bf v}+\lambda\nabla\cdot \left(\zeta{\bf v}\right)+a\nabla\cdot\Delta {\bf v}=0,\nonumber&&\\
(1-d\Delta)({\bf v})_{t}+r_{2}\nabla\zeta+\frac{\lambda}{2}\nabla |{\bf v}|^{2}=0,&&\;\;\label{BB24b}
\end{eqnarray}
As in \cite{A}, we apply the operator $\Lambda^{s}=(I-\Delta)^{s/2}$ to (\ref{BB24b}), multiply the first equation by $(1-\gamma)\Lambda^{s}\zeta$, the second by $r_{1}\Lambda^{s}{\bf v}+a\Lambda^{s}\Delta{\bf v}$ (Hadamard sense), integrate over $\Omega$ and add the resulting integrals. Using the periodic boundary conditions and after some computations, all this leads to
\begin{eqnarray*}
0&=&\int_{\Omega}\frac{(1-\gamma)}{2}\frac{d}{dt}|\Lambda^{s}\zeta|^{2}+\frac{r_{1}}{2}\frac{d}{dt}|\Lambda^{s}{\bf v}|^{2}+\frac{r_{1}d}{2}\frac{d}{dt}|\Lambda^{s}\nabla{\bf v}|^{2}\\
&&-\frac{a}{2}\frac{d}{dt}|\Lambda^{s}\nabla{\bf v}|^{2}-\frac{ad}{2}\frac{d}{dt}|\Lambda^{s}\Delta{\bf v}|^{2}d\Omega\\
&&+\int_{\Omega}\left((1-\gamma)\lambda\Lambda^{s}\zeta\Lambda^{s}(\nabla\cdot(\zeta{\bf v}))+\frac{\lambda r_{1}}{2}\Lambda^{s}{\bf v}\Lambda^{s}\nabla\cdot|{\bf v}|^{2}\right.\\
&&\left.+\frac{a\lambda}{2}\Lambda^{s}\Delta{\bf v}\Lambda^{s}(\nabla\cdot|{\bf v}|^{2})\right)d\Omega.
\end{eqnarray*}
Now we make use of Kato-Ponce estimates, the embedding (\ref{*1}),
and Young's inequalities to obtain similar estimates to those of (\cite{A}; Theorem 3.3)
\begin{eqnarray*}
|\langle \Lambda^{s}\zeta,\Lambda^{s}(\nabla\cdot\zeta{\bf v})\rangle|&\lesssim& ||\zeta||_{s}^{3}+||{\bf v}||_{s+2}^{3},\\
|\langle \Lambda^{s}{\bf v},\Lambda^{s}\nabla\cdot|{\bf v}|^{2}\rangle|&\lesssim&||{\bf v}||_{s+1}^{3},\\
|\langle \Lambda^{s}\Delta{\bf v},\Lambda^{s}(\nabla\cdot|{\bf v}|^{2})\rangle|&\lesssim& ||{\bf v}||_{s+2}^{3}.
\end{eqnarray*}
Defining
\begin{eqnarray*}
Y(t)&=&\int_{\Omega}{(1-\gamma)}|\Lambda^{s}\zeta|^{2}+{\alpha_{1}}|\Lambda^{s}{\bf v}|^{2}+{\alpha_{1}d}|\Lambda^{s}\nabla{\bf v}|^{2}\\
&&-a|\Lambda^{s}\nabla{\bf v}|^{2}-{ad}|\Lambda^{s}\Delta{\bf v}|^{2}d\Omega,
\end{eqnarray*}
then we have
\begin{eqnarray*}
Y'(t)\lesssim Y(t)^{3/2},
\end{eqnarray*}
providing the bound of $Y(t), t\in [0,T], T$ small enough. Using that $a<0, d>0$, this leads to a priori bound of $(\zeta,{\bf v})$ in
$$L^{\infty}(0,T,H^{s})\times (L^{\infty}(0,T,H^{s+2}))^{2}.$$ The justification of the a priori estimate and the existence and uniqueness of solutions can be proved following the arguments in (\cite{A}, Sec. 2.2); see also \cite{Lions,BonaSmith}
\end{proof}
\begin{remark}
As in the case of the corresponding ivp, \cite{DMS2007}, the periodic ivp (\ref{BB21}), (\ref{BB22}) when $b=d$ admits a Hamiltonian structure
\begin{eqnarray*}
\frac{\partial}{\partial t}\begin{pmatrix} \zeta\\{\bf v}\end{pmatrix}+{\mathcal J}\delta H(\zeta,{\bf v})=0,
\end{eqnarray*}
with
\begin{eqnarray*}
{\mathcal J}=\begin{pmatrix}(I-b\Delta)^{-1}&0\\{\bf 0}&(I-b\Delta)^{-1}\end{pmatrix}\begin{pmatrix}0&{\rm div}\\ \nabla&{\bf 0}\end{pmatrix},
\end{eqnarray*}
and
\begin{eqnarray}
H(\zeta,{\bf v})=\int_{\Omega}\left(-c'|\nabla\zeta|^{2}-a|\nabla {\bf v}|^{2}+(r_{1}+\lambda\zeta)|{\bf v}|^{2}+r_{2}\zeta^{2}\right)d{\bf x},\label{Ham}
\end{eqnarray} 
where $|\nabla{\bf v}|^{2}=|\nabla v_{1}|^{2}+|\nabla v_{2}|^{2}=(\partial_{x}v_{1})^{2}+(\partial_{y}v_{1})^{2}+(\partial_{x}v_{2})^{2}+(\partial_{y}v_{2})^{2}$.
\end{remark}
\section{Spectral Fourier-Galerkin semidiscretization}
\label{sec3}
In this section we will analyze the discretization in space of (\ref{BB21}), (\ref{BB22}) with a spectral Fourier-Galerkin method. Let $N\geq 1$ be an integer and
\begin{eqnarray*}
S_{N}={\rm span}\{\phi_{{\bf k}}(x,y), {\bf k}=(k_{x},k_{y}), k_{j}\in\mathbb{Z}, -N\leq k_{j}\leq N, j=x,y\},
\end{eqnarray*}
where
\begin{eqnarray*}
\phi_{\bf k}(x,y)=e^{2\pi i{\bf x}\cdot{\bf k}},\quad {\bf x}=(x,y), {\bf x}\cdot{\bf k}=k_{x}x+k_{y}y.
\end{eqnarray*}
Let $P=P_{N}$ denote the $L^{2}$ projection operator onto $S_{N}$: for $u\in L^{2}(\Omega)$
\begin{eqnarray*}
P_{N}u=\sum_{ -N\leq k_{x},k_{y}\leq N}\widehat{u}_{\bf k}\phi_{\bf k},\quad
\widehat{u}_{\bf k}=\frac{1}{(2\pi)^{2}}\int_{\Omega}u({\bf x})\phi_{\bf k}({\bf x})d{\bf x}.
\end{eqnarray*}
In the sequel we will make use of the following estimates, also valid in $\Omega$, \cite{CHQZ}.
\begin{itemize}
\item[(i)] For integers $0\leq l\leq m$ and $u\in H_{p}^{m}(\Omega)$
\begin{eqnarray}
||u-P_{N}u||_{H^{l}}&\lesssim &N^{l-m}||u||_{H^{m}},\label{estim1}\\
|u-P_{N}u|_{\infty}&\lesssim & N^{1/2-m}||u||_{H^{m}},\quad m\geq 1.\label{estim2}
\end{eqnarray}
\item[(ii)] (Inverse inequalities.) For integers $0\leq l\leq m$, $\phi\in S_{N}$
\begin{eqnarray}
||\phi||_{H^{m}}\lesssim N^{m-l}||\phi||_{H^{l}}.\label{estim3}
\end{eqnarray}
\end{itemize}
The spectral Galerkin semidiscretization of (\ref{BB21}), (\ref{BB22}) is defined as follows: Let $T>0$. We seek for mappings $\zeta^{N}, v_{j}^{N}:[0,T]\rightarrow S_{N}, j=1,2,$ such that if ${\bf v}^{N}=(v_{1}^{N},v_{2}^{N})$, for $0\leq t\leq T$, $\varphi\in S_{N}, {\bf \chi}=(\chi_{1},\chi_{2})\in S_{N}\times S_{N}$, it holds that
 \begin{eqnarray}
(\zeta_{t}^{N},\varphi)+r_{1}(\nabla\cdot{\bf v}^{N},\varphi)+\lambda (\nabla\cdot(\zeta^{N}{\bf v}^{N},\varphi)+a(\Delta\nabla\cdot{\bf v}^{N},\varphi)&&\nonumber\\
-b(\Delta\zeta_{t}^{N},\varphi)=0,&&\label{BB31a}\\
\langle{\bf v}^{N}_{t},{\bf \chi}\rangle+r_{2}\langle\nabla\zeta^{N},{\bf \chi}\rangle+\frac{\lambda}{2}\langle\nabla |{\bf v}^{N}|^{2},{\bf \chi}\rangle+c'\langle\Delta\nabla\zeta^{N},{\bf \chi}\rangle-d\langle\Delta{\bf v}^{N}_{t},{\bf \chi}\rangle=0,&&\label{BB31b}\\
\zeta^{N}(0)=P_{N}\zeta_{0},\quad {\bf v}^{N}(0)=P_{N}{\bf v}_{0}.&&\label{BB31c}
\end{eqnarray}
The ode system (\ref{BB31a})-(\ref{BB31c}) can be written as
\begin{eqnarray}
(I-b\Delta)\zeta_{t}^{N}=-{\rm div}A_{N},\;
(I-b\Delta)\partial_{t}{\bf v}^{N}=-\nabla B_{N},\label{BB33a}
\end{eqnarray}
where
\begin{eqnarray}
A_{N}&=&r_{1}{\bf v}^{N}+\lambda \widetilde{P_{N}}(\zeta^{N}{\bf v}^{N}+a\Delta {\bf v}^{N},\nonumber\\
B_{N}&=&r_{2}\zeta^{N}+\frac{\lambda}{2}P_{N}(|{\bf v}^{N}|^{2})+c'\Delta\zeta^{N},\label{BB33b}
\end{eqnarray}
and $\widetilde{P_{N}}({\bf v})=(P_{N}(v_{1}),P_{N}v_{2})$ if ${\bf v}=(v_{1},v_{2})\in S_{N}\times S_{N}$. Using classical ode theory, the corresponding ivp
has a unique solution, locally in time. 
The Fourier representation is, for ${\bf k}=(k_{x},k_{y}), -N\leq k_{x},k_{y}\leq N$
\begin{eqnarray}
(1+b|{\bf k}|^{2})\partial_{t}\widehat{\zeta^{N}}({\bf k},t)+ir_{1}(k_{x}+k_{y})\widehat{{\bf v}^{N}}({\bf k},t)&&\nonumber\\
+i\lambda(k_{x}\widehat{\zeta^{N}v_{1}^{N}}({\bf k},t)+k_{y}\widehat{\zeta^{N}v_{2}^{N}}({\bf k},t)
-ia|{\bf k}|^{2}(k_{x}\widehat{v_{1}^{N}}({\bf k},t)+k_{y}\widehat{v_{2}^{N}}({\bf k},t))=0,&&\label{BB32a}\\
(1+d|{\bf k}|^{2})\partial_{t}\widehat{{\bf v}^{N}}({\bf k},t)+ir_{2}\begin{pmatrix}k_{x}\\k_{y}\end{pmatrix}\widehat{\zeta^{N}}({\bf k},t)&&\nonumber\\
+i\frac{\lambda}{2}\begin{pmatrix}k_{x}\\k_{y}\end{pmatrix}\widehat{|{\bf v}^{N}|^{2}}({\bf k},t)-ic'{|{\bf k}|^{2}}\begin{pmatrix}k_{x}\\k_{y}\end{pmatrix}\widehat{\zeta^{N}}({\bf k},t)=0,&&\label{BB32b}\\
\widehat{\zeta^{N}}({\bf k},0)=\widehat{\zeta_{0}}({\bf k}),\quad \widehat{{\bf v}^{N}}({\bf k},0)=\widehat{{\bf v}_{0}}({\bf k}).&&\label{BB32c}
\end{eqnarray}
\subsection{Conservation properties}
\begin{proposition}
\label{prop31} Assume $b=d$. While the solution $(\zeta^{N},{\bf v}^{N})$ of (\ref{BB31a})-(\ref{BB31c}) exists, then 
\begin{eqnarray*}
\frac{d}{dt}H(\zeta^{N},{\bf v}^{N})=0,
\end{eqnarray*}
where $H$ is given by (\ref{Ham}).
\end{proposition}
\begin{proof}
Using integration by parts and the periodic boundary conditions, we compute
\begin{eqnarray}
\frac{d}{dt}H(\zeta^{N},{\bf v}^{N})&=&\int_{\Omega}\left(\zeta_{t}^{N}\left(c'\Delta \zeta^{N}+\frac{\lambda}{2}|{\bf v}^{N}|^{2}+r_{2}\zeta^{N}\right)\right.\nonumber\\
&&+\left.\partial_{t}{\bf v}^{N}\cdot\left(a\Delta{\bf v}^{N}+r_{1}{\bf v}^{N}+\lambda(\zeta^{N}{\bf v}^{N}\right)\right)d{\bf x}.\label{BB33c}
\end{eqnarray}
Now, from (\ref{BB33a}), (\ref{BB33b}), and the property
\begin{eqnarray*}
(P_{N}\varphi,\psi)=(\varphi,\psi),\quad \varphi\in L^{2},\; \psi\in S_{N},
\end{eqnarray*}
then (\ref{BB33c}) has the form
\begin{eqnarray*}
\frac{d}{dt}H(\zeta^{N},{\bf v}^{N})=-\int_{\Omega}\left(\left((I-b\Delta)^{-1}{\rm div}A_{N}\right)B_{N}+A_{N}\cdot (I-b\Delta)^{-1}\nabla B_{N}\right)d{\bf x},
\end{eqnarray*}
which, using again integration by parts and the periodic boundary conditions, vanishes.
\end{proof}
\subsection{Error estimates}
In this section we estimate the error of the semidiscrete scheme (\ref{BB31a}), (\ref{BB31b}) for some well-posed B/B systems, extending the results obtained in \cite{DDS0} for the one-dimensional case.

Let $\theta=\zeta^{N}-P_{N}\zeta, \rho=P_{N}\zeta-\zeta$, so that $\zeta^{N}-\zeta=\theta+\rho$, and ${\bf \xi}=(\xi_{1},\xi_{2})={\bf v}^{N}-\widetilde{P_{N}}{\bf v}, {\bf \sigma}=(\sigma_{1},\sigma_{2})=\widetilde{P_{N}}{\bf v}-{\bf v}$, so that ${\bf v}^{N}-{\bf v}={\bf \xi}+{\bf \sigma}$. After some computations, the system for $\theta$ and ${\bf \xi}$ takes the form, for $\varphi\in S_{N}, {\bf \chi}=(\chi_{1},\chi_{2})\in S_{N}\times S_{N}$
\begin{eqnarray}
(\theta_{t},\varphi)+a(\Delta\nabla\cdot\xi,\varphi)-b(\Delta\theta_{t},\varphi)&=&-r_{1}(\nabla\cdot\xi,\varphi)-(\nabla\cdot A,\varphi),\label{BB34a}\\
\langle\xi_{t}\chi\rangle+c'\langle\Delta\nabla\theta,\chi\rangle-d\langle\Delta\xi_{t},\chi\rangle&=&-r_{2}\langle\nabla\theta,\chi\rangle-\langle\nabla B,\chi\rangle,\label{BB34b}
\end{eqnarray}
where
\begin{eqnarray}
A&=&\lambda(\zeta^{N}{\bf v}^{N}-\zeta{\bf v})=\lambda\left(\rho{\bf v}+\zeta{\bf \sigma}+\theta{\bf v}+\zeta\xi+\theta\sigma+\rho\xi+\rho\sigma+\theta\xi\right),\label{BB34c}\\
B&=&\frac{\lambda}{2}(|{\bf v}^{N}|^{2}-{\bf v}|^{2})=\lambda\left({\bf v}\cdot\sigma+{\bf v}\cdot\xi+\sigma\cdot\xi+\frac{|\xi|^{2}+|\sigma|^{2}}{2}\right).\label{BB34d}
\end{eqnarray}
According to the well-posedness analysis in section \ref{sec2}, and classified from \cite{A}, we will obtain convergence results in the following cases:
\begin{itemize}
\item Weakly dispersive B/B systems (cf. Theorem \ref{theo21}):
\begin{itemize}
\item[(C1)] $b,d>0, a=c=0$.
\item[(C2)] $b,d>0, a,c<0$.
\item[(C3)] $b,d>0, a=0, c<0$.
\item[(C4)] $b,d>0, a<0, c=0$.
\end{itemize}
\item Other admissible B/B systems (cf. Theorem \ref{theo22}):
\begin{itemize}
\item[(C5)] $b=0, d>0, a, c<0$.
\item[(C6)] $b=0,d>0, a<0, c=0$.
\item[(C7)] $b>0,d=0, a=c=0$.
\item[(C8)] $b=0,d>0, a=c=0$.
\end{itemize}
\end{itemize}
\begin{proposition}
\label{prop32}
Assume that the solution $(\zeta,{\bf v})=(\zeta,v_{1},v_{2})$ of (\ref{BB21})-(\ref{BB22}) satisfies that $\zeta,v_{1},v_{2}\in C^{1}(0,T,H^{\mu}), \mu\geq 1.$. Let $a,b,c,d$ be in one of the cases (C1)-(C4). Then
\begin{eqnarray}
\max_{0\leq t\leq T}\left(||\zeta^{N}-\zeta||+||{\bf v}^{N}-{\bf v}||\right)&\lesssim&N^{-\mu},\label{BB35a}\\
\max_{0\leq t\leq T}\left(||\zeta^{N}-\zeta||_{1}+||{\bf v}^{N}-{\bf v}||_{1}\right)&\lesssim&N^{1-\mu},\label{BB35b}
\end{eqnarray}
\end{proposition}
\begin{proof}
We first consider the case (C1). Taking $\varphi=\theta, \chi=\xi$ in (\ref{BB34a}), (\ref{BB34b}) and after some computations, we have
\begin{eqnarray*}
\frac{d}{dt}\left(\frac{||\theta||^{2}}{2}+b\frac{||\nabla\theta||^{2}}{2}\right)&=&r_{1}\langle\xi,\nabla\theta\rangle+\langle A,\nabla\theta\rangle,\\
\frac{d}{dt}\left(\frac{||\xi||^{2}}{2}+d\frac{||\nabla\xi_{1}||^{2}+||\nabla\xi_{2}||^{2}}{2}\right)&=&r_{2}(\theta,\nabla\cdot\xi)+(B,\nabla\cdot\xi).
\end{eqnarray*}
We add the previous equations to have
\begin{eqnarray*}
&&\frac{1}{2}\frac{d}{dt}\left(||\theta||^{2}+||\xi||^{2}+b||\nabla\theta||^{2}+d(||\nabla\xi_{1}||^{2}+||\nabla\xi_{2}||^{2}\right)\nonumber\\
&=&r_{1}\langle\xi,\nabla\theta\rangle+\langle A,\nabla\theta\rangle+r_{2}(\theta,\nabla\cdot\xi)+(B,\nabla\cdot\xi).\label{BB36a}
\end{eqnarray*}
According to (\ref{BB34c}), we now estimate each of the terms in $\langle A,\nabla\theta\rangle$, as follows: Using (\ref{estim1}), we have
\begin{eqnarray*}
\left|\langle\rho{\bf v},\nabla\theta\rangle\right|&=&\left|\int_{\Omega}\rho(v_{1}\theta_{x}+v_{2}\theta_{y})d{\bf x}\right|\leq |{\bf v}|_{\infty}||\rho||||\nabla\theta||\\
&\lesssim & N^{-2\mu}+||\nabla\theta||^{2}.
\end{eqnarray*}
Similarly
\begin{eqnarray*}
\left|\langle\zeta{\bf \sigma},\nabla\theta\rangle\right|&\leq& |{\zeta}|_{\infty}||\sigma||||\nabla\theta||\lesssim  N^{-2\mu}+||\nabla\theta||^{2},\\
\left|\langle\theta{\bf v},\nabla\theta\rangle\right|&=&\left|\int_{\Omega}\theta(v_{1}\theta_{x}+v_{2}\theta_{y})d{\bf x}\right|\leq |{\bf v}|_{\infty}|\int_{\Omega}(\theta\theta_{x}+\theta\theta_{y})d{\bf x}\\
&\leq & |{\bf v}|_{\infty}||\theta||_{1}^{2},\\
\left|\langle\zeta{\bf \xi},\nabla\theta\rangle\right|&\leq& |{\zeta}|_{\infty}||\xi||||\nabla\theta||\lesssim  N^{-2\mu}+||\nabla\theta||^{2}.
\end{eqnarray*}
Using (\ref{estim2}), we have $|\sigma|_{\infty}\leq C$ and therefore
\begin{eqnarray*}
\left|\langle\theta{\bf \sigma},\nabla\theta\rangle\right|\lesssim |{\sigma}|_{\infty}||\theta||_{1}^{2}\lesssim ||\theta||_{1}^{2}.
\end{eqnarray*}
Similarly
\begin{eqnarray*}
\left|\langle\rho{\bf \xi},\nabla\theta\rangle\right|&\leq &|{\rho}|_{\infty}||\xi||||\nabla\theta||\lesssim  ||\xi||^{2}+||\nabla\theta||^{2}.\\
\left|\langle\rho{\bf \sigma},\nabla\theta\rangle\right|&\leq &|{\rho}|_{\infty}||\sigma||||\nabla\theta||\lesssim  N^{-2\mu}+||\nabla\theta||^{2}.
\end{eqnarray*}
The same argument of continuity of $\theta$ as that of \cite{DDS0} is used to define $t_{N}$, $0<t_{N}\leq T$, as the maximal time for which the solution of (\ref{BB31a})-(\ref{BB31c}) exists and satisfies
\begin{eqnarray}
|\theta|_{\infty}\leq 1,\; 0\leq t\leq t_{N}.\label{maximal}
\end{eqnarray}
Then, for $0\leq t\leq t_{N}$
\begin{eqnarray*}
\left|\langle \theta\xi,\nabla\theta\rangle\right|\leq |\theta|_{\infty}||\xi||||\nabla\theta||\leq \frac{1}{2}\left(||\xi||^{2}+||\nabla\theta||^{2}\right).
\end{eqnarray*}
Therefore, if $0\leq t\leq t_{N}$
\begin{eqnarray}
\left|\langle A,\nabla\theta\rangle\right|\lesssim N^{-2\mu}+||\xi||^{2}+||\theta||_{1}^{2}.\label{BB36b}
\end{eqnarray}
On the other hand
\begin{eqnarray*}
\left|(\theta,\nabla\cdot\xi)\right|&\leq& ||\theta||||\nabla\cdot\xi||\leq \frac{1}{2}\left(||\theta||^{2}+||\xi||_{1}^{2}\right),\\
\left|\langle \xi,\nabla\theta\rangle\right|&\leq& \frac{1}{2}\left(||\xi||^{2}+||\nabla\theta||^{2}\right).
\end{eqnarray*}
For the terms in $(B,\nabla\cdot\xi)$, we have, using (\ref{estim1})-(\ref{estim3}), the following estimates:
\begin{eqnarray*}
\left|(v\cdot\sigma,\nabla\cdot\xi)\right|&\leq&|v|_{\infty}||\sigma||\left(||\nabla\xi_{1}||+||\nabla\xi_{2}||\right)\\
&\lesssim&N^{-2\mu}+||\xi ||_{1}^{2},\\
\left|(v\cdot\xi,\nabla\cdot\xi)\right|&\lesssim&||\xi||_{1}^{2},\\
\left|(\sigma\cdot\xi,\nabla\cdot\xi)\right|&\lesssim&||\xi||_{1}^{2},\\
\left|\frac{1}{2}(|\sigma|^{2},\nabla\cdot\xi)\right|&\leq&\frac{1}{2}|\sigma|_{\infty}||\sigma||||\xi||_{1}\lesssim N^{-2\mu}+||\xi||_{1}^{2}.
\end{eqnarray*}
On the other hand, from periodicity
\begin{eqnarray*}
\left|\frac{1}{2}(|\xi|^{2},\nabla\cdot\xi)\right|&=&\frac{1}{2}\int_{\Omega}(\xi_{1}^{2}+\xi_{2}^{2})(\partial_{x}\xi_{1}+\partial_{y}\xi_{2})d{\bf x}\\
&=&\frac{1}{2}\left(\int_{0}^{L}\left(\underbrace{\int_{0}^{L}\xi_{1}^{2}\partial_{x}\xi_{1}dx}_{=0}\right)dy+\left(\underbrace{\int_{0}^{L}\xi_{2}^{2}\partial_{y}\xi_{2}dy}_{=0}\right)dx\right)\\
&&+\frac{1}{2}\int_{\Omega}(\xi_{1}^{2}\partial_{y}\xi_{2}+\xi_{2}^{2}\partial_{x}\xi_{1})d{\bf x}=\int_{\Omega}\xi_{1}\xi_{2}(\partial_{y}\xi_{1}+\partial_{x}\xi_{2})d{\bf x}.
\end{eqnarray*}
Then, using for $\xi_{1}$ and $\xi_{2}$ similar arguments to those used for $\theta$ we have, for some $t_{N}$ and $0\leq t\leq t_{N}$
\begin{eqnarray}
|\xi_{j}(t)|_{\infty}\lesssim 1,\; j=1,2.\label{maximal2}
\end{eqnarray}
This implies
\begin{eqnarray*}
\left|\frac{1}{2}(|\xi|^{2},\nabla\cdot\xi)\right|\lesssim |\xi_{1}|_{\infty}||\xi_{2}||_{1}^{2}+ |\xi_{2}|_{\infty}||\xi_{1}||_{1}^{2}\lesssim ||\xi||_{1}^{2}.
\end{eqnarray*}
Therefore, if $0\leq t\leq t_{N}$
\begin{eqnarray}
|(B,\nabla\cdot\xi)|\lesssim N^{-2\mu}+||\xi||_{1}^{2}.\label{BB36c}
\end{eqnarray}
Hence, since $b,d>0$, we have, for $0\leq t\leq t_{N}$
\begin{eqnarray*}
\frac{d}{dt}\left(||\theta||_{1}^{2}+||\xi||_{1}^{2}\right)\lesssim N^{-2\mu}+||\theta||_{1}^{2}+||\xi||_{1}^{2}.
\end{eqnarray*} 
By Gronwall's lemma and (\ref{BB31c}) we have, for $0\leq t\leq t_{N}$ and some constant $C=C(T)$, there holds
\begin{eqnarray}
||\theta||_{1}+||\xi||_{1}\leq C N^{-\mu}.\label{BB36d}
\end{eqnarray}
The usual argument, exposed in e.~g. \cite{DMS2007,DDS0}, proves that $t_{N}$ is not maximal in (\ref{maximal}), (\ref{maximal2}); we may therefore take $t_{N}=T$ and from (\ref{BB36a}), (\ref{BB36b}), and (\ref{BB36c}), then (\ref{BB36d}) holds for $0\leq t\leq T$, leading to (\ref{BB35a}) and, along with (\ref{estim1}), to (\ref{BB35b}).

We now consider the case (C2). Taking $\varphi=\theta, \chi=\xi$ in (\ref{BB34a}), (\ref{BB34b}), integrating by parts,  and after some computations, we have
\begin{eqnarray}
\frac{d}{dt}\left(\frac{||\theta||^{2}}{2}+b\frac{||\nabla\theta||^{2}}{2}\right)-a\langle\nabla(\nabla\cdot\xi,\nabla\theta\rangle&=&r_{1}\langle\xi,\nabla\theta\rangle\nonumber\\
&&+\langle A,\nabla\theta\rangle,\label{BB37a}\\
\frac{d}{dt}\left(\frac{||\xi||^{2}}{2}+d\frac{||\nabla\xi_{1}||^{2}+||\nabla\xi_{2}||^{2}}{2}\right)+c'\langle\nabla\theta,\nabla\nabla\cdot\xi\rangle&=&r_{2}(\theta,\nabla\cdot\xi)\nonumber\\
&&+(B,\nabla\cdot\xi).\label{BB37b}
\end{eqnarray}
We multiply (\ref{BB37a}) by $-c'$, (\ref{BB37b}) by $-a$, and add the resulting equations to obtain
\begin{eqnarray*}
\frac{1}{2}\frac{d}{dt}\left(|c'|||\theta||^{2}+|a|||\xi||^{2}+b|c'|||\nabla\theta||^{2}\right.&&\\
\left.+d|a|\left(||\nabla\xi_{1}||^{2}+||\nabla\xi_{2}||^{2}\right)\right)&=&-c'r_{1}\langle\xi,\nabla\theta\rangle-c'\langle A,\nabla\theta\rangle\\
&&-ar_{2}(\theta,\nabla\cdot\xi)-a(B,\nabla\cdot\xi).
\end{eqnarray*}
Following the same arguments as in the case (C1) the estimates (\ref{BB35a}), (\ref{BB35b}) follow.

We now consider the case (C3). Equation (\ref{BB34a}) is written in the form
\begin{eqnarray}
(I-b\Delta)\theta_{t}=-r_{1}\nabla\cdot\xi-P_{N}(\nabla\cdot A).\label{BB37c}
\end{eqnarray}
Note that the operator $T_{\alpha}=(I-\alpha\Delta)^{-1}, \alpha>0$, is well defined in $H^{j}, j\in\mathbb{R}$: if $f\in H^{j-2}$ then $T_{\alpha}f\in H^{j}$ and, from it Fourier representation, it holds that
\begin{eqnarray}
||T_{\alpha}f||_{j}\lesssim ||f||_{j-2}.\label{BB37d}
\end{eqnarray}
Then we can write (\ref{BB37c}) in the form
\begin{eqnarray*}
\theta_{t}=-r_{1}T_{b}(\nabla\cdot\xi)-T_{b}P_{N}(\nabla\cdot A),
\end{eqnarray*}
and since $\partial_{x},\partial_{y}$ commute with $T_{b}$ and $P_{N}$, from (\ref{BB37d}) there holds
\begin{eqnarray}
||\theta_{t}||_{1}&\leq&|r_{1}|||T_{b}(\nabla\cdot\xi)||_{1}+||T_{b}P_{N}(\nabla\cdot A)||_{1}\nonumber\\
&\lesssim&||\xi||+||A||.\label{BB37dd}
\end{eqnarray}
Now, from (\ref{BB34c}), we have
\begin{eqnarray}
||A||&\lesssim&|{\bf v}|_{\infty}||\rho||+|\zeta|_{\infty}||{\bf v}||+|{\bf v}|_{\infty}||\theta||\nonumber\\
&&+|\zeta|_{\infty}||\xi||+|\sigma|_{\infty}||\theta||+|\rho|_{\infty}||\xi||\nonumber\\
&&+|\rho|_{\infty}||\sigma||+(|\xi_{1}|_{\infty}+|\xi_{2}|_{\infty})||\theta||.\label{BB37e}
\end{eqnarray}
Since $\xi_{1}(0)=\xi_{2}(0)=0$ and using continuity, there is $t_{N}\in (0,T]$ a maximal time for which the solution of the semidiscrete ivp exists and saisfies
\begin{eqnarray}
|\xi_{j}(t)|_{\infty}\leq 1,\quad 0\leq t\leq t_{N}, \; j=1,2.\label{BB37f}
\end{eqnarray}
Then, using (\ref{estim1}), (\ref{estim2}), and (\ref{BB37f}) we have, from (\ref{BB37e}) and for $0\leq t\leq t_{N}$
\begin{eqnarray*}
||A||\lesssim N^{-\mu}+||\theta||+||\xi||.
\end{eqnarray*}
Then, (\ref{BB37dd}) implies that
\begin{eqnarray}
||\theta_{t}||\lesssim N^{-\mu}+||\theta||+||\xi||,\; 0\leq t\leq t_{N}.\label{BB37g}
\end{eqnarray}
Similarly, we write (\ref{BB34b}) as
\begin{eqnarray*}
\xi_{t}=-c'T_{d}\Delta\nabla\theta-r_{2}T_{d}\nabla\theta-T_{d}\widetilde{P}_{N}(\nabla B),
\end{eqnarray*}
(where the operators $T_{d}\Delta, T_{d}$ acts on the components) which, using (\ref{BB37d}), leads to
\begin{eqnarray}
||\xi_{t}||&\leq&|c'|||T_{d}\Delta\nabla\theta||+|r_{2}|||T_{d}\nabla\theta||+||T_{d}\widetilde{P}_{N}(\nabla B)||\nonumber\\
&\lesssim&||\theta||_{1}+||\theta||+||B||.\label{BB37h}
\end{eqnarray}
From (\ref{BB34d}), we have
\begin{eqnarray*}
||B||&\lesssim&|{\bf v}|_{\infty}||\sigma||+|{\bf v}|_{\infty}||\xi||+|\sigma|_{\infty}||\xi||\\
&&+|\sigma|_{\infty}||\sigma||+(|\xi_{1}|_{\infty}+|\xi_{2}|_{\infty})||\xi ||.
\end{eqnarray*}
As above, using (\ref{estim1}), (\ref{estim2}), and (\ref{BB37f}) we have, for $0\leq t\leq t_{N}$
\begin{eqnarray*}
||B||\lesssim N^{-\mu}+||\theta||_{1}+||\xi||,
\end{eqnarray*}
and therefore, from (\ref{BB37h})
\begin{eqnarray}
||\xi_{t}||\lesssim N^{-\mu}+||\theta||_{1}+||\xi||,\; 0\leq t\leq t_{N}.\label{BB37i}
\end{eqnarray}
Thus, (\ref{BB37g}) and (\ref{BB37i}) lead to
\begin{eqnarray*}
||\theta_{t}||_{1}+||\xi_{t}||\lesssim N^{-\mu}+||\theta||_{1}+||\xi||,\; 0\leq t\leq t_{N}.
\end{eqnarray*}
Since $\theta(0)=\xi(0)=0$, Gronwall's lemma implies that
\begin{eqnarray*}
||\theta(t)||_{1}+||\xi(t)||\leq CN^{-\mu},\; 0\leq t\leq t_{N},
\end{eqnarray*}
for some constant $C$, dependent on $T$. Then, since $\mu\geq 1$ and using (\ref{estim3}), we infer, for $N$ large enough, that $t{N}$ is not maximal satisfying (\ref{BB37f}) and we may take $t_{N}=T$ so that
\begin{eqnarray*}
||\theta(t)||_{1}+||\xi(t)||\leq CN^{-\mu},\; 0\leq t\leq T,
\end{eqnarray*}
and from this point, the same reasoning as that in case (C1) can be used to obtain the conclusion. We also note that for the case (C4) we can argue in a similar way to that of (C3). 
\end{proof}

\begin{proposition}
\label{prop33}
Assume that the solution $(\zeta,{\bf v})=(\zeta,v_{1},v_{2})$ of (\ref{BB21})-(\ref{BB22}) satisfies that $\zeta,v_{1},v_{2}\in C^{1}(0,T,H^{\mu}), \mu\geq 1$. Let $a,b,c,d$ be in one of the cases (C5)-(C8). Then, for $\mu>3/2$
\begin{eqnarray}
\max_{0\leq t\leq T}\left(||\zeta^{N}-\zeta||+||{\bf v}^{N}-{\bf v}||_{1}\right)&\lesssim&N^{1-\mu},\label{bb35c}
\end{eqnarray}
\end{proposition}
\begin{proof}
We first consider the case (C5). Taking $\varphi=\theta, \chi=\xi$ in (\ref{BB34a}), (\ref{BB34b}), integrating by parts,  and after some computations, we have
\begin{eqnarray}
\frac{d}{dt}\frac{||\theta||^{2}}{2}-a\langle\nabla(\nabla\cdot\xi,\nabla\theta\rangle&=&-r_{1}(\nabla\cdot\xi,\theta)\nonumber\\
&&-(\nabla\cdot A,\theta),\label{BB38a}\\
\frac{d}{dt}\left(\frac{||\xi||^{2}}{2}+d\frac{||\nabla\xi_{1}||^{2}+||\nabla\xi_{2}||^{2}}{2}\right)+c'\langle\nabla\theta,\nabla\nabla\cdot\xi\rangle&=&r_{2}(\theta,\nabla\cdot\xi)\nonumber\\
&&+(B,\nabla\cdot\xi).\label{BB38b}
\end{eqnarray}
We multiply (\ref{BB38a}) by $-c'$, (\ref{BB38b}) by $-a$, and add the resulting equations to obtain
\begin{eqnarray}
\frac{1}{2}\frac{d}{dt}\left(|c'|||\theta||^{2}+|a|||\xi||^{2}+d|a|\left(||\nabla\xi_{1}||^{2}+||\nabla\xi_{2}||^{2}\right)\right)&=&c'r_{1}(\nabla\cdot\xi,\theta)+c'(\nabla\cdot A,\theta)\nonumber\\
&&-ar_{2}(\theta,\nabla\cdot\xi)\nonumber\\
&&-a(B,\nabla\cdot\xi).\label{BB38c}
\end{eqnarray}
We now estimate the right-hand side of (\ref{BB38c}). First we have
\begin{eqnarray*}
\left|(\nabla\cdot\xi,\theta)\right|\leq \frac{1}{2}\left(||\theta||^{2}+||\nabla\cdot\xi||^{2}\right).
\end{eqnarray*}
We use the definition of $A$ in (\ref{BB34c}), the hypotheses on the exact solution,  and (\ref{estim1}) for the following estimates:
\begin{eqnarray*}
\left|(\nabla\cdot(\rho {\bf v}),\theta)\right|&\leq&\left(|\rho|_{\infty}||{\bf v}||_{1}+|{\bf v}|_{\infty}||\rho||_{1}\right)||\theta||\\
&\leq&||\rho||_{1}(||{\bf v}||_{1}+|{\bf v}|_{\infty})||\theta||\\
&\lesssim&||\rho||_{1}||\theta||\lesssim N^{2(1-\mu)}+||\theta||^{2}.
\end{eqnarray*}
Similarly
\begin{eqnarray*}
\left|(\nabla\cdot(\zeta {\bf \sigma}),\theta)\right|&\leq&\left(|\zeta|_{\infty}(||{\sigma_{1}}||_{1}+||\sigma_{2}||_{1})+||\zeta||_{1}(||{\sigma_{1}}||_{1}+||\sigma_{2}||_{1})\right)||\theta||\\
&\lesssim&(||{\sigma_{1}}||_{1}+||\sigma_{2}||_{1})||\theta||\\
&\lesssim& N^{2(1-\mu)}+||\theta||^{2}.
\end{eqnarray*}
And
\begin{eqnarray*}
\left|(\nabla\cdot(\zeta {\bf \xi}),\theta)\right|&\leq&\left(|\zeta|_{\infty}(||{\xi_{1}}||_{1}+||\xi_{2}||_{1})+||\zeta||_{1}(||{\xi_{1}}||_{1}+||\xi_{2}||_{1})\right)||\theta||\\
&\lesssim&(||{\xi_{1}}||_{1}+||\xi_{2}||_{1})||\theta||\\
&\lesssim& ||\xi||_{1}^{2}+||\theta||^{2},\\
\left|(\nabla\cdot(\rho {\bf \xi}),\theta)\right|&\lesssim& ||\xi||_{1}^{2}+||\theta||^{2}.
\end{eqnarray*}

Using integration by parts, we have
\begin{eqnarray*}
(\nabla\cdot(\theta{\bf v}),\theta)=-\langle\theta{\bf v},\nabla\theta\rangle=-\langle {\bf v},\nabla\left(\frac{\theta^{2}}{2}\right)\rangle=(\nabla\cdot{\bf v},\frac{\theta^{2}}{2}).
\end{eqnarray*}
Therefore, since $\mu>3/2$
\begin{eqnarray*}
\left|(\nabla\cdot(\theta{\bf v}),\theta)\right|=\left|(\nabla\cdot{\bf v},\frac{\theta^{2}}{2})\right|\lesssim ||\theta||^{2}.
\end{eqnarray*}
Similarly, from integration by parts, (\ref{estim1})-(\ref{estim3}), and since $\mu>3/2$
\begin{eqnarray*}
\left|(\nabla\cdot(\theta{\bf \sigma}),\theta)\right|&\lesssim&||\theta||^{2},\\
\left|(\nabla\cdot(\rho{\bf \sigma}),\theta)\right|&\lesssim&N^{\frac{3}{2}-2\mu}||\theta||\lesssim N^{-\mu}||\theta||\lesssim N^{-2\mu}+||\theta||^{2}.
\end{eqnarray*}
Since $\theta(0)=0$, using continuity, there is some $t_{N}, 0<t_{N}\leq T$, the maximal value of $t$ for which the solution of (\ref{BB31a})-(\ref{BB31c}) exists and satisfies
\begin{eqnarray}
|\theta(t)|_{\infty}\leq 1,\; 0\leq t\leq t_{N}.\label{BB38d}
\end{eqnarray}
By (\ref{BB38d}), for $0\leq t\leq t_{N}$, it holds that
\begin{eqnarray*}
\left|(\nabla\cdot(\theta{\bf \xi}),\theta)\right|=\left|(\nabla\cdot{\bf \xi},\frac{\theta^{2}}{2})\right|\lesssim|\theta|_{\infty}||\theta||||\nabla\cdot\xi||\lesssim ||\theta||^{2}+||\xi||_{1}^{2}.
\end{eqnarray*}
Therefore, from (\ref{BB34c}) and the previous estimates, we have, for $0\leq t\leq t_{N}$
\begin{eqnarray*}
\left|(\nabla\cdot A,\theta)\right|\lesssim N^{2(1-\mu)}+||\theta||^{2}+||\xi||_{1}^{2}.
\end{eqnarray*}
On the other hand, we can estimate the terms $(\theta,\nabla\cdot\xi)$ and $(B,\nabla\cdot\xi)$ as in Proposition \ref{prop32} to have
\begin{eqnarray*}
\left|(\theta,\nabla\cdot\xi)\right|+\left|(B,\nabla\cdot\xi)\right|\lesssim N^{-2\mu}+||\theta||^{2}+||\xi||_{1}^{2}.
\end{eqnarray*}
We apply the previous estimates to (\ref{BB38c}) leading to, for $0\leq t\leq t_{N}$
\begin{eqnarray*}
\frac{1}{2}\frac{d}{dt}\left(||\theta||^{2}+||\xi||_{1}^{2}\right)\lesssim N^{2(1-\mu)}+||\theta||^{2}+||\xi||_{1}^{2}.
\end{eqnarray*}
From (\ref{BB31c}), Gronwall's lemma applies and it holds that
\begin{eqnarray}
||\theta||+||\xi||_{1}\leq C N^{1-\mu},\label{BB38e}
\end{eqnarray}
for some constant $C=C(T)$. Using (\ref{estim3}) we have
\begin{eqnarray*}
|\theta|_{\infty}\lesssim N^{1/2}||\theta||,
\end{eqnarray*}
and, from (\ref{BB38e}) and since $\mu>3/2$, then $t_{N}$ is not maximal in (\ref{BB38d}) if $N$ is large enough; we may take $t_{N}=T$ and (\ref{BB38e}) leads to (\ref{bb35c}).

We now consider the case (C6). Taking $\varphi=\theta$ in (\ref{BB34a}) and  integrating by parts, we have, while the semidiscrete approximation exists
\begin{eqnarray}
\frac{d}{dt}\frac{||\theta||^{2}}{2}=-r_{1}(\nabla\cdot\xi,\theta)-(\nabla\cdot A,\theta)+a\langle\nabla(\nabla\cdot\xi),\nabla\theta\rangle,\label{BB39a}
\end{eqnarray}
Taking $\chi=\xi+a\nabla(\nabla\cdot\xi)$ in (\ref{BB34b}) and integrating by parts, after some computations
\begin{eqnarray}
\frac{1}{2}\frac{d}{dt}\left(||\xi||^{2}+(|a|+d)(||\nabla\xi_{1}||^{2}+||\nabla\xi_{2}||^{2})\right.&&\nonumber\\
\left.+|a|d||\nabla(\nabla\cdot\xi)||^{2}\right)&=&-r_{1}\langle\nabla\theta,\xi\rangle\label{BB39b}\\
&&-ar_{2}\langle\nabla\theta,\nabla(\nabla\cdot\xi)\rangle\nonumber\\
&&-\langle\nabla B,\xi+a\nabla(\nabla\cdot\xi)\rangle.\nonumber
\end{eqnarray}

We multiply (\ref{BB39a}) by $r_{2}$ and add to (\ref{BB39b}) to obtain
\begin{eqnarray}
\frac{1}{2}\frac{d}{dt}\left(r_{2}||\theta||^{2}+||\xi||^{2}+(|a|+d)||\nabla\xi||^{2}\right.&&\nonumber\\
\left.+|a|d||\nabla(\nabla\cdot\xi)||^{2}\right)&=&-r_{1}r_{2}(\nabla\cdot\xi,\theta)-r_{2}(\nabla\cdot A,\theta)\nonumber\\
&&-r_{2}\langle\nabla\theta,\xi\rangle\nonumber\\
&&-\langle\nabla B,\xi+a\nabla(\nabla\cdot\xi)\rangle.\label{BB39c}
\end{eqnarray}
Let $t_{N}$ be the maximal time in $(0,T]$ for which the semidiscrete approximation exists and satisfies (\ref{BB38d}) and (\ref{BB37f}). The terms $(\nabla\cdot\xi,\theta)$ and $(\nabla\cdot A,\theta)$ are estimated as in the case (C5) and therefore, for $0\leq t\leq t_{N}$
\begin{eqnarray*}
\left|(\nabla\cdot\xi,\theta)\right|+\left|(\nabla\cdot A,\theta)\right|\lesssim N^{2(1-\mu)}+||\theta||^{2}+||\xi||_{1}^{2}.
\end{eqnarray*}
On the other hand, notice that
\begin{eqnarray*}
\left|\langle\nabla\theta,\xi\rangle\right|=\left|(\theta,\nabla\cdot\xi)\right|\leq \frac{1}{2}\left(||\theta||^{2}+||\xi||_{1}^{2}\right).
\end{eqnarray*}
Note also that as, in Proposition \ref{prop32} we have
\begin{eqnarray*}
\left|(\theta,\nabla\cdot\xi)\right|+\left|(B,\nabla\cdot\xi)\right|\lesssim N^{-2\mu}+||\theta||^{2}+||\xi||_{1}^{2}.
\end{eqnarray*}
We now estimate the terms in 
\begin{eqnarray*}
\langle\nabla B,\nabla(\nabla\cdot\xi)\rangle=\frac{\lambda}{2}\langle\nabla({\bf v}\cdot\sigma+{\bf v}\cdot\xi+\sigma\cdot\xi+\frac{|\xi|^{2}+|\sigma|^{2}}{2},\nabla(\nabla\cdot\xi)\rangle
\end{eqnarray*}
Using integration by parts, (\ref{estim1}), and since $\mu>3/2$, we have:
\begin{eqnarray*}
\left|\langle\nabla({\bf v}\cdot\sigma),\nabla(\nabla\cdot\xi)\rangle\right|&\lesssim& N^{2(1-\mu)}+||\nabla(\nabla\cdot\xi||^{2}.\\
\left|\langle({\bf v}\cdot\xi),\nabla(\nabla\cdot\xi)\rangle\right|&\lesssim&||\xi||_{1}||\nabla(\nabla\cdot\xi||\lesssim ||\xi||_{2}^{2}.\\
\left|\langle{\bf \sigma}\cdot\xi,\nabla(\nabla\cdot\xi)\rangle\right|&\lesssim&||\xi||_{1}||\nabla(\nabla\cdot\xi||\lesssim ||\xi||_{2}^{2}.\\
\left|\langle\nabla |{\bf v}|^{2},\nabla(\nabla\cdot\xi)\rangle\right|&\lesssim&N^{\frac{3}{2}-2\mu}||\nabla(\nabla\cdot\xi||\lesssim N^{-2\mu}+||\nabla(\nabla\cdot\xi||^{2}.
\end{eqnarray*}
Using (\ref{BB37f}), for $0\leq t\leq t_{N}$, we have
\begin{eqnarray*}
\left|\langle\nabla |{\bf \xi}|^{2},\nabla(\nabla\cdot\xi)\rangle\right|\lesssim (|\xi_{1}|_{\infty}+|\xi_{2}|_{\infty})||\nabla\xi||||\nabla(\nabla\cdot\xi)||^{2}\lesssim ||\xi||_{2}^{2}.
\end{eqnarray*}
Now we apply the previous estimates to (\ref{BB39c}) yielding, for $0\leq t\leq t_{N}$
\begin{eqnarray*}
\frac{d}{dt}\left(||\theta||^{2}+||\xi||_{2}^{2}\right)\lesssim N^{2(1-\mu)}+||\theta||^{2}+||\xi||_{2}^{2}.
\end{eqnarray*}
Gronwall's lemma and (\ref{BB31c}) imply, for $0\leq t\leq t_{N}$
\begin{eqnarray}
||\theta||+||\xi||_{2}\leq C N^{(1-\mu)},\label{BB39d}
\end{eqnarray}
for some constant $C=C(T)$. From (\ref{estim3}) and since $\mu>3/2$, then $t_{N}$ is not maximal in (\ref{BB38d}) and (\ref{BB37f}). Then we may take $t_{N}=T$, (\ref{BB39d}) holds for $0\leq t\leq T$, leading, as before, the estimate (\ref{bb35c}).

The cases (C7) and (C8) can be proved using similar arguments. With no loss of generality we assume (C8). We take $\varphi=\theta, \chi=\xi$ in (\ref{BB34a}), (\ref{BB34b}), integrate by parts, and add the resulting equations to have
\begin{eqnarray}
\frac{1}{2}\frac{d}{dt}\left(||\theta||^{2}+||\xi||^{2}+d\left(||\nabla\xi_{1}||^{2}+||\nabla\xi_{2}||^{2}\right)\right)&=&-r_{1}(\nabla\cdot\xi,\theta)-(\nabla\cdot A,\theta)\nonumber\\
&&+r_{2}(\theta,\nabla\cdot\xi)\nonumber\\
&&+(B,\nabla\cdot\xi).\label{BB310a}
\end{eqnarray}
Each term on the right-hand side of (\ref{BB310a}) can be estimated as in case (C5), and the arguments there imply (\ref{bb35c}).
\end{proof}
\section{Numerical experiments}
\label{sec4}
In order to illustrate and complete the numerical analysis performed in section \ref{sec3}, some numerical experiments are here presented. They are focused on the following points: (i) The introduction of scheme for the time integration of the semidiscrete system (\ref{BB31a}), (\ref{BB31b}); (ii) The validation of the full discretization and illustration of the estimates derived in Propositions \ref{prop32} and \ref{prop33}; (iii) The simulation of line solitary wave solutions of some of the B/B systems.
\subsection{Full discretization}
For purposes of implementation, the periodic ivp (\ref{BB31a})-(\ref{BB31c}) will will be posed on an interval $\Omega_{L}=(-L,L)^{2}$. The corresponding semidiscrete system, after the spectral formulation, will have a Fourier representation of the form, cf. (\ref{BB32a})-(\ref{BB32c})
\begin{eqnarray}
\frac{d}{dt}\begin{pmatrix}\widehat{\zeta^{N}}\\\widehat{{\bf v}^{N}}\end{pmatrix}(\widetilde{\bf k},t)&=&f(\widehat{\zeta^{N}},\widehat{{\bf v}^{N}},\widetilde{\bf k},t)\label{BB41a}\\
&=&-i|{\widetilde{\bf k}}|\mathcal{A}({\widetilde{\bf k}})+F(\widehat{\zeta^{N}},\widehat{{\bf v}^{N}},\widetilde{\bf k},t)\nonumber
\end{eqnarray}
where $\widetilde{{\bf k}}=(\widetilde{k}_{x},\widetilde{k}_{y})=\frac{\pi}{L}{\bf k}$, ${\bf k}=(k_{x},k_{y}), -N\leq k_{x},k_{y}\leq N$, 
$\mathcal{A}, F$ given by (\ref{BB24}), and initial conditions
\begin{eqnarray}
\widehat{\zeta^{N}}(\widetilde{\bf k},0)=\widehat{\zeta_{0}}(\widetilde{\bf k}),\quad \widehat{{\bf v}^{N}}(\widetilde{\bf k},0)=\widehat{{\bf v}_{0}}(\widetilde{\bf k}).\label{BB42c}
\end{eqnarray}
The ivp (\ref{BB41a}), (\ref{BB42c}) is approximated at times $t_{m}=m\Delta t$, $m=0,\ldots,M$, with step size $\Delta t$, up to a final time $T=M\Delta t$, by the implicit midpoint rule (IMR), which is second-order accurate and has well-known stability and geometric properties (like symplecticity and symmetry) guaranteeing its efficiency for long time simulations. The implementation of the full discretization is made as follows. First, after some change of interval, the Fourier Galerkin approximation is formulated in collocation form in $\Omega_{L}$. This means that $(\zeta^{N},{\bf v}^{N})$ in (\ref{BB41a}) is represented by the values at some uniform grid $(x_{j},y_{k})$ of collocation points $x_{j}=-L+jh, y_{k}=-L+kh, j,k=0,\ldots,N-1$. The matrices
\begin{eqnarray*}
W^{N}(t)=\begin{pmatrix}Z^{N}(t)\\{\bf V}^{N}(t)\end{pmatrix}=\begin{pmatrix}\zeta^{N}(x_{j},y_{k},t)\\v_{1}^{N}(x_{j},y_{k},t)\\v_{2}^{N}(x_{j},y_{k},t)\end{pmatrix}_{j,k=0}^{N-1},
\end{eqnarray*} 
will satisfy the ode system $\frac{d}{dt}W^{N}=f(W^{N})$, that is
\begin{eqnarray}
(I+b\Delta_{N})\frac{d}{dt}Z^{N}+r_{1}\nabla_{N}\cdot{\bf V}^{N}+\lambda\nabla_{N}\cdot(Z^{N}.{\bf V}^{N})&&\nonumber\\
+a\Delta_{N}\nabla_{N}\cdot{\bf V}^{N}&=&0,\nonumber\\
(I+d\Delta_{N})\frac{d}{dt}{\bf V}^{N}+r_{2}\nabla_{N}Z^{N}+\frac{\lambda}{2}\nabla_{N}|{\bf V}^{N}|^{2}&&\nonumber\\
+c\Delta_{N}\nabla_{N}Z^{N}&=&0,\label{*2}
\end{eqnarray}
where
\begin{eqnarray*}
\Delta_{N}=D_{N,x}^{2}+D_{N,y}^{2},\; \nabla_{N}=\begin{pmatrix}D_{N,x}\\D_{N,y}\end{pmatrix},
\end{eqnarray*}
being $D_{N,j}$ the $N\times N$ Fourier pseudospectral differentiation matrix in the $j$ direction, $j=x,y$, and if ${\bf V}^{N}=\begin{pmatrix}V_{1}^{N}\\V_{2}^{N}\end{pmatrix}$
\begin{eqnarray*}
Z^{N}.{\bf V}^{N}=\begin{pmatrix}Z^{N}.V_{1}^{N}\\Z^{N}.V_{2}^{N}\end{pmatrix},\;
|{\bf V}^{N}|^{2}=|V_{1}^{N}|^{2}+|V_{2}^{N}|^{2},
\end{eqnarray*}
with the dot operation in Hadamard sense. The system (\ref{*2}) is then integrated numerically by the IMR with the formulas
\begin{eqnarray}
W^{*}=W^{n}+\frac{\Delta t}{2}f(W^{N}),\quad W^{n+1}=2W^{*}-W^{n},\label{*3}
\end{eqnarray}
where $W^{n}$ approximates $W^{N}(t_{n}), n=0,\ldots,M$. The implementation is made in the Fourier space for the $2D$ discrete Fourier coefficients of $W^{n}$ (which allows to extend the formulation to different partitions for $x$ and $y$) and the classical fixed point iteration is used to solve numerically the internal stage in (\ref{*3})

\subsection{Numerical experiments}
\label{sec42}
The performance of the fully discrete method will be here illustrated with some numerical experiments. For simplicity, we will focus on the BBM-BBM case ($b=d>0, a=c=0$) and the computations will be divided in several groups of experiments, concerning:
\begin{itemize}
\item The validation and performance of the full discretization.
\item The dynamics of the internal wave system from initially localized, zero-velocity pulses.
\item The generation and resolution of line solitary waves.
\end{itemize}
Some of the results can be compared with similar experiments performed in \cite{DMS2007} for the case of surface wave propagation and using different boundary conditions.

The first group of experiments is a comparison with exact solutions of (\ref{BB21}) in the BBM-BBM case, given by the solitary waves of the form (cf. \cite{DMS2007})
\begin{eqnarray}
\zeta(x,y,t)&=&A_{1}{\rm sech}^{2}(A(x-c_{s}t-x_{0}))+A_{2}{\rm sech}^{4}(A(x-c_{s}t-x_{0})),\nonumber\\
v_{1}(x,y,t)&=&B_{1}{\rm sech}^{2}(A(x-c_{s}t-x_{0})),\nonumber\\
v_{2}(x,y,t)&=&0,\label{BB43}
\end{eqnarray}
where, for $c_{s}\neq 0$, and after some computations, the constants $A,A_{1},A_{2},B$ satisfy
\begin{eqnarray*}
&&c_{s}^{2}(1-4bA^{2})^{2}=r_{1}r_{2},\quad B_{1}=\frac{20bc_{s}}{\lambda}A^{2},\\
&&A_{2}=\frac{B_{1}}{2r_{2}}(12bc_{s}A^{2}-\lambda B_{1}),\quad A_{1}=\frac{c_{s}A_{2}(1-16bA^{2})}{\lambda B_{1}-6bc_{s}A^{2}}.
\end{eqnarray*}
Due to the exponential decay of (\ref{BB43}) as $|x|\rightarrow\infty$ we may consider an interval with large enough $L$ and approximate the ivp for (\ref{BB1}) by the periodic ivp (\ref{BB21}) in $\Omega_{L}$ with initial data for (\ref{BB22}) given by (\ref{BB43}) at $t=0$. The periodic version of (\ref{BB43}), \cite{BChen}, is solution of (\ref{BB21}) and in this sense the fully discrete solution can be considered as approximation of (\ref{BB43}) in $\Omega_{L}$, being possible to check the accuracy through the corresponding errors. This is illustrated in Tables \ref{tav1} and \ref{tav2}, corresponding to the $L^{2}$ and $L^{\infty}$ norms, respectively, for a solution (\ref{BB43})  with $c_{s}=5/2$ at final time $T=5$, in the case  $\gamma=0.5,\delta=0.9$ and taking $L=16, N=256$.
\begin{table}[htbp]
\begin{tabular}{c|c|c|c|c|}
$\Delta t$& $\zeta$-Error&Rate&$v_{1}$-Error&Rate\\
\hline
$2.5\times 10^{-2}$&$2.9955\times 10^{-1}$&&$2.8148\times 10^{-1}$&\\
$1.25\times 10^{-2}$&$7.5015\times 10^{-2}$&$1.997$&$7.0542\times 10^{-2}$&$1.997$\\
$6.25\times 10^{-3}$&$1.8778\times 10^{-2}$&$1.998$&$1.7660\times 10^{-3}$&$1.998$\\
$3.125\times 10^{-3}$&$4.7571\times 10^{-3}$&$1.981$&$4.4767\times 10^{-3}$&$1.980$\\
\hline
\end{tabular}
\caption{$L^{2}$ errors at $T=5$ and temporal convergence rates for the method (\ref{*3}) with respect to (\ref{BB43}) with $c_{s}=5/2$, $x_{0}=-10$, $\gamma=0.5,\delta=0.9$, $L=16, N=256$.\label{tav1}}
\end{table}

\begin{table}[htbp]
\begin{tabular}{c|c|c|c|c|}
$\Delta t$& $\zeta$-Error&Rate&$v_{1}$-Error&Rate\\
\hline
$2.5\times 10^{-2}$&$5.2223\times 10^{-2}$&&$3.9235\times 10^{-2}$&\\
$1.25\times 10^{-2}$&$1.3049\times 10^{-2}$&$2.000$&$9.8321\times 10^{-3}$&$1.997$\\
$6.25\times 10^{-3}$&$3.2542\times 10^{-3}$&$2.004$&$2.4596\times 10^{-3}$&$1.999$\\
$3.125\times 10^{-3}$&$8.0540\times 10^{-4}$&$2.015$&$6.1509\times 10^{-4}$&$2.000$\\
\hline
\end{tabular}
\caption{$L^{\infty}$ errors at $T=5$ and temporal convergence rates for the method (\ref{*3}) with respect to (\ref{BB43}) with $c_{s}=5/2$, $x_{0}=-10$, $\gamma=0.5,\delta=0.9$, $L=16, N=256$.\label{tav2}}
\end{table}
The corresponding rates of convergence show the temporal order; due to the regularity of the solution, spectral convergence in space is attained. Figure \ref{BBFig1} represents the $\zeta$ component of the numerical solution at times $t=1,4,8$ while Figure \ref{BBFig2}(a) shows the corresponding cross sections at $y=0$.
\begin{figure}[htbp]
\centering
\subfigure
{\includegraphics[width=0.45\columnwidth]{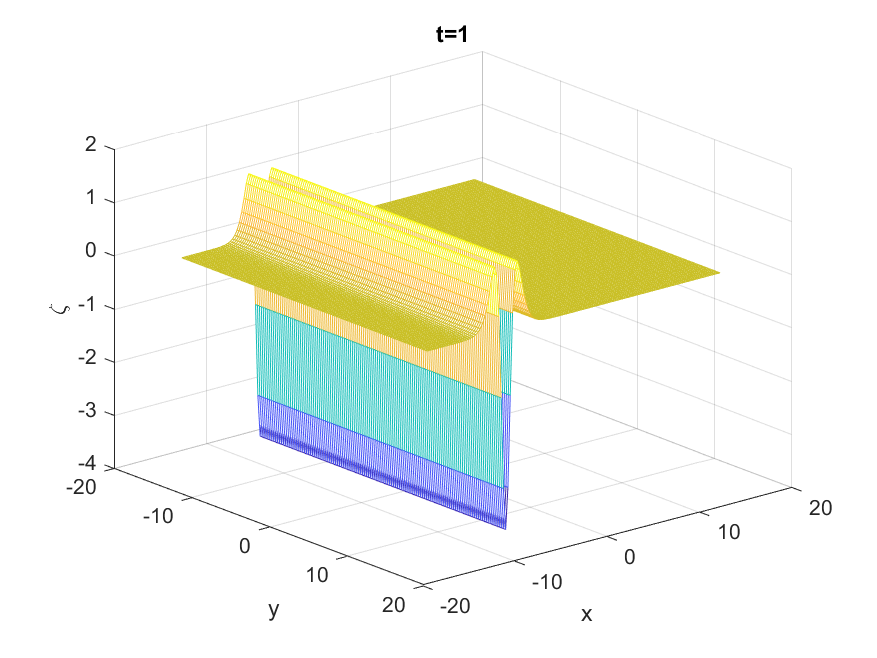}}
\subfigure
{\includegraphics[width=0.45\columnwidth]{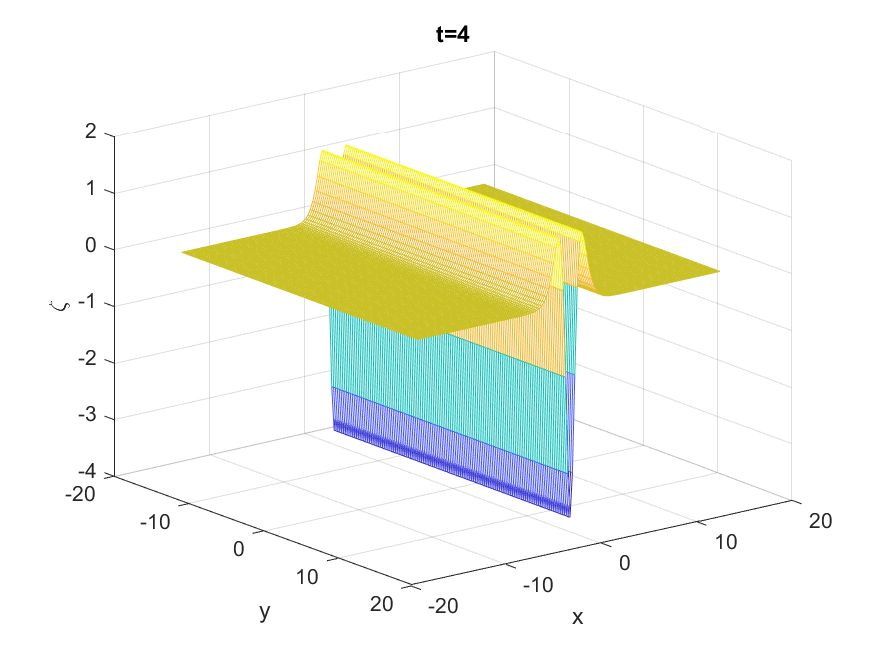}}
\subfigure
{\includegraphics[width=0.45\columnwidth]{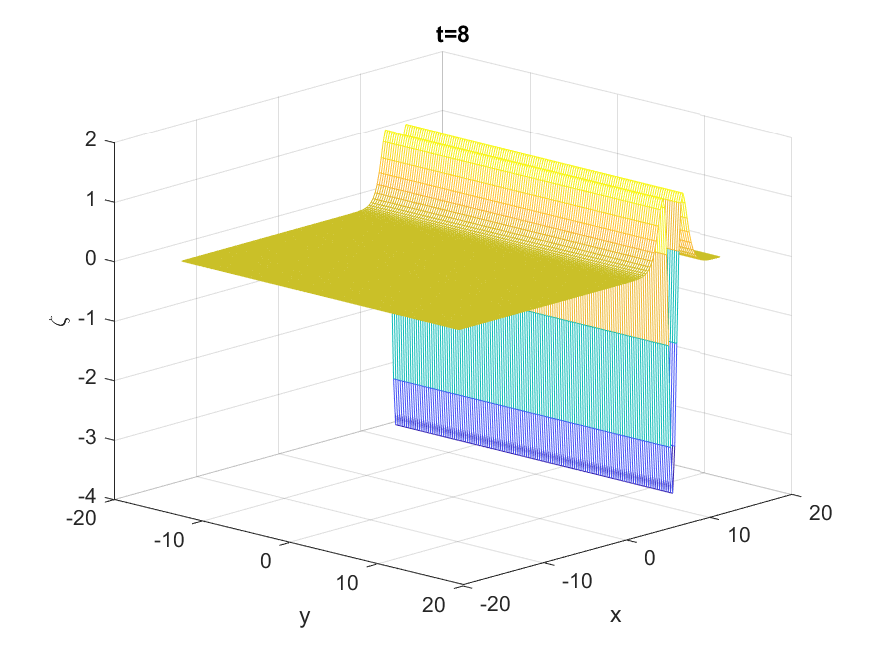}}
\caption{$\zeta$ component of the numerical solution at times $t=1,4,8$.}
\label{BBFig1}
\end{figure}
\begin{figure}[htbp]
\centering
\subfigure[]
{\includegraphics[width=0.45\columnwidth]{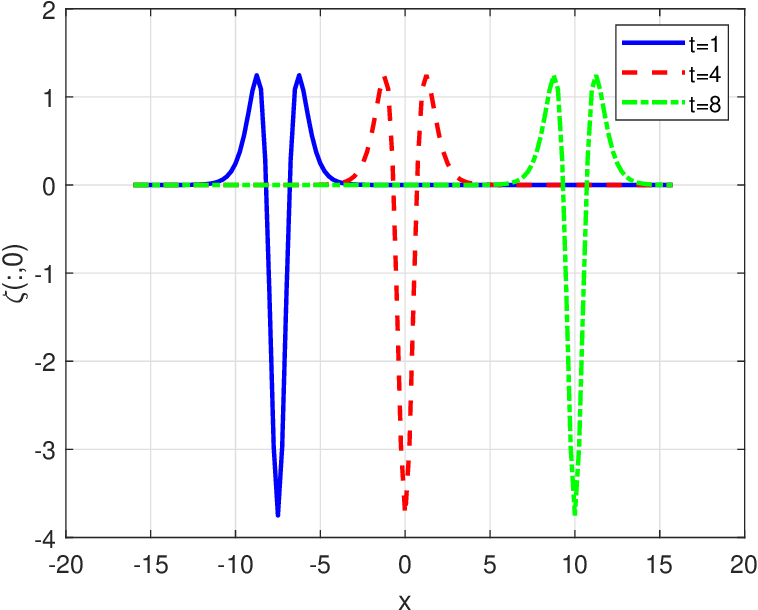}}
\subfigure[]
{\includegraphics[width=0.45\columnwidth]{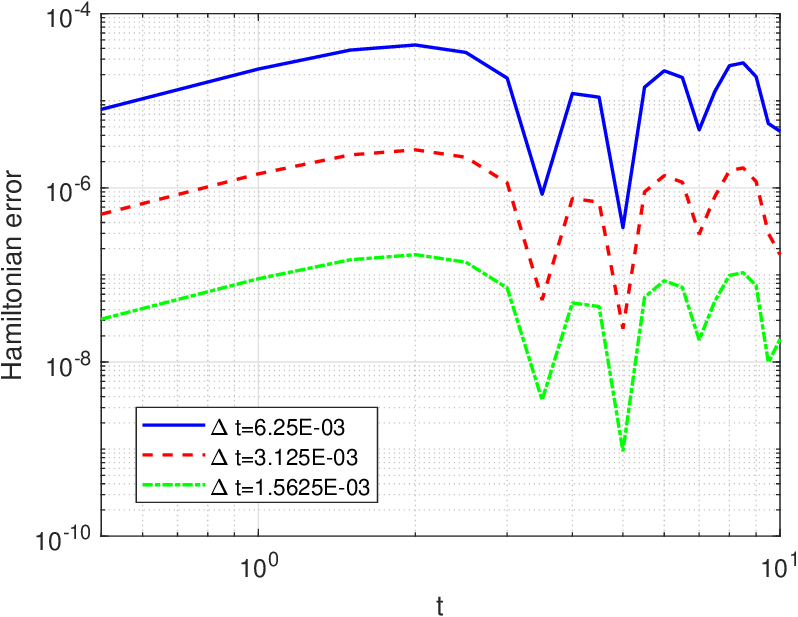}}
\caption{(a) $\zeta$ component of the numerical solution at times $t=1,4,8$ and $y=0$; (b) Time behaviour of the Hamiltonian error for several time steps in loglog scale.}
\label{BBFig2}
\end{figure}
A second indication of accuracy is given by Figure \ref{BBFig2}(b), which shows the evolution of the error in the Hamiltonian (\ref{Ham}) (measured by the corresponding discrete version implemented with quadrature rules, the collocation representation of the numerical solution and FFT techniques). The results obtained from different time steps show that error behaves like the order of the temporal discretization $O(\Delta t^{2})$ and it is bounded up to the final time of simulation, as consequence of the geometric properties of the time discretization.

The accuracy of the numerical method, checked in the previous numerical experiments, will be used here to illustrate, by computational means, the dynamics of (\ref{BB1}) from the evolution of initially localized waves of initial zero velocity. The results can be compared with those from similar experiments performed in \cite{DMS2007} for two models of free surface ($b=d>0, a=c=0$ and $b=d>0, a=0, c<0$) and using reflective boundary conditions (of homogeneous Neumann type for $\zeta$ and homogeneous Dirichlet for ${\bf v}$). In our case, we will fix again $\gamma=0.5, \delta=0.9$, focus on the BBM-BBM system, and consider a long enough interval with $L=64$ in order to minimize the influence of the boundary conditions.

The first experiments of this group study computationally the dynamics from initial Gaussian pulses. Figure \ref{BBFig3} shows the evolution of the $\zeta$ component of the numerical solution obtained from the initial condition
\begin{eqnarray}
\zeta_{0}(x,y)&=&0.1 e^{-(x^{2}+y^{2})/5},\nonumber\\
v_{1}^{0}(x,y)&=&v_{2}^{0}(x,y)=0.\label{BB44}
\end{eqnarray}
As mentioned in \cite{DMS2007}, the solution will satisfy, for all $t$
$$\zeta(x,y,t)=\zeta(-x-y,t),\; v_{1}(x,y,t)=-v_{1}(-x,-y,t),\;
v_{2}(x,y,t)=-v_{2}(-x,-y,t).$$ 
\begin{figure}[htbp]
\centering
\subfigure
{\includegraphics[width=0.45\columnwidth]{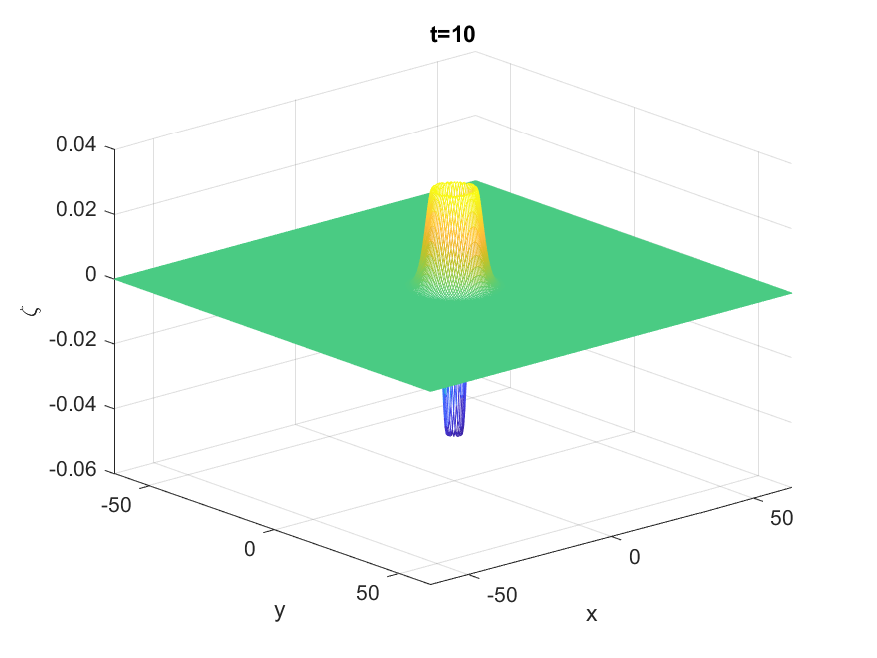}}
\subfigure
{\includegraphics[width=0.45\columnwidth]{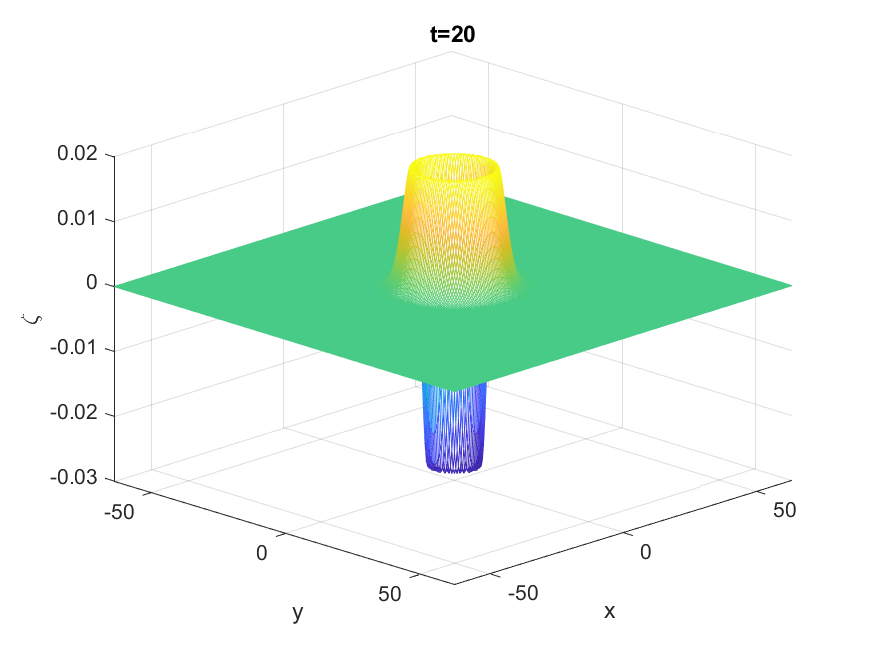}}
\subfigure
{\includegraphics[width=0.45\columnwidth]{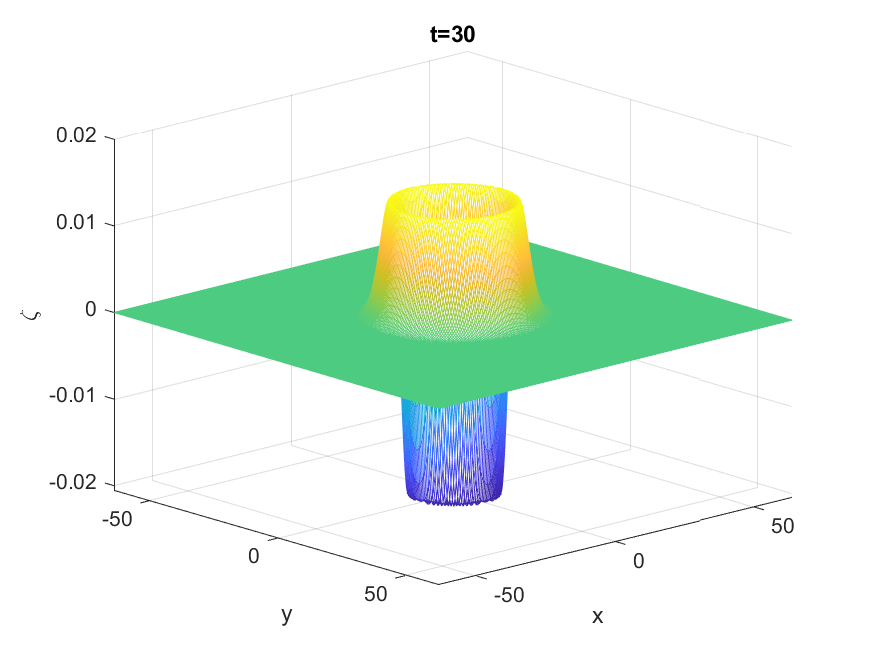}}
\subfigure
{\includegraphics[width=0.45\columnwidth]{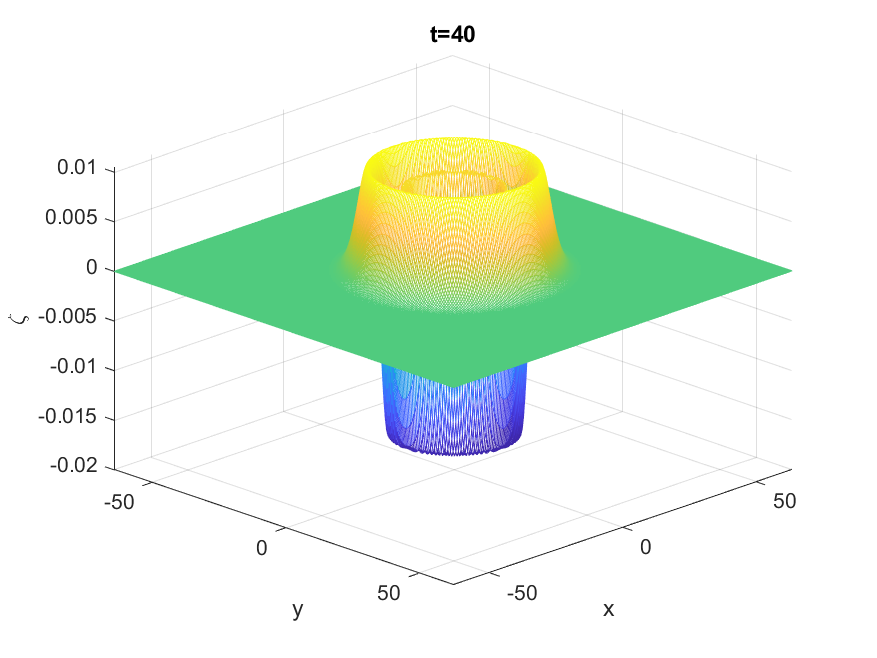}}
\caption{$\zeta$ component of the numerical solution from (\ref{BB44}) at times $t=10,20,30,40$.}
\label{BBFig3}
\end{figure}
\begin{figure}[htbp]
\centering
\subfigure
{\includegraphics[width=0.45\textwidth]{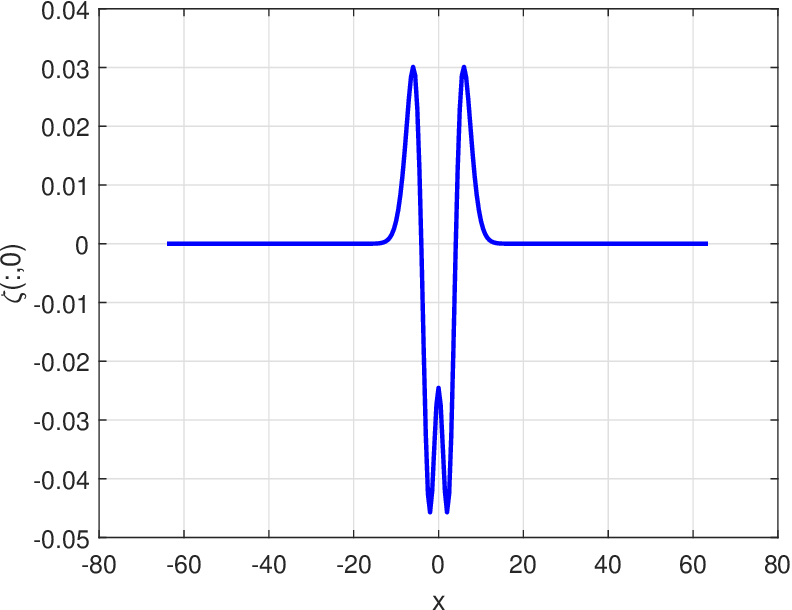}}
\subfigure
{\includegraphics[width=0.45\textwidth]{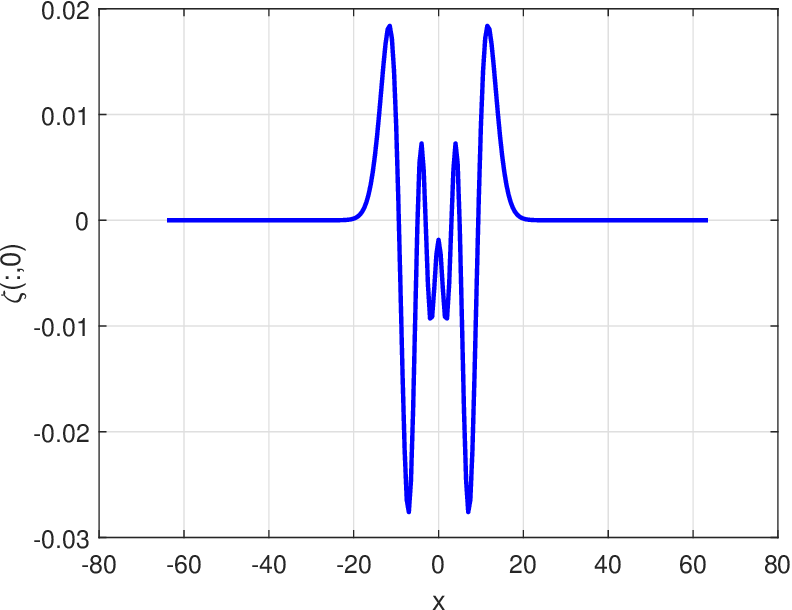}}
\subfigure
{\includegraphics[width=0.45\textwidth]{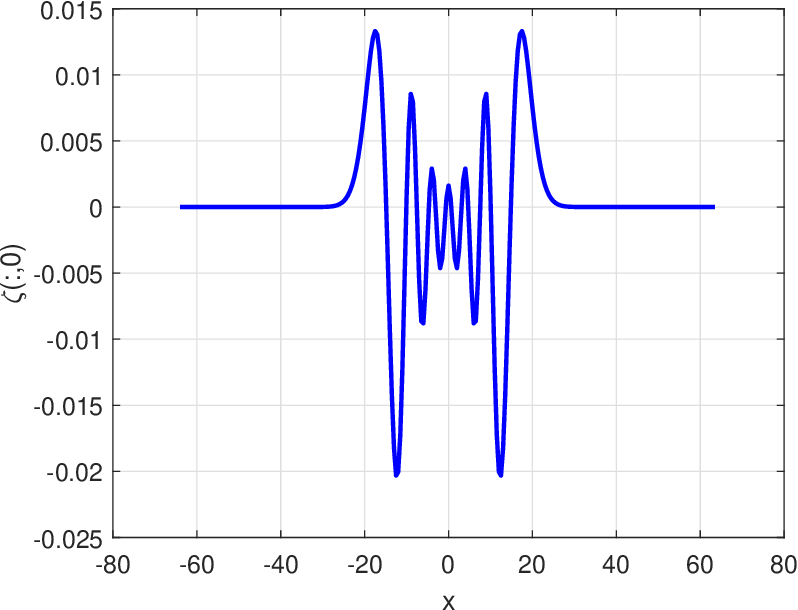}}
\subfigure
{\includegraphics[width=0.45\textwidth]{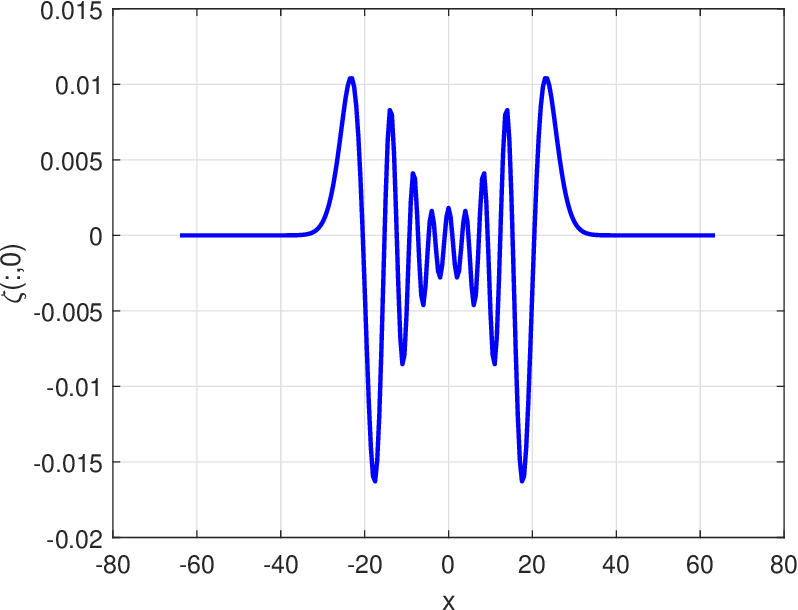}}
\caption{$\zeta$ component of the numerical solution from (\ref{BB44}) at times $t=10,20,30,40$ and $y=0$.}
\label{BBFig4}
\end{figure}
From Figure \ref{BBFig3} we observe the symmetric expansion of the wave in a dispersive, oscillatory way. This is also suggested in Figure \ref{BBFig4}, which shows the one-dimensional cross section of the $\zeta$ component of the numerical solution at $y=0$.

When the initial data is not symmetric, a different behaviour may be observed. Figure \ref{BBFig5} shows the evolution of the $\zeta$ component of the numerical solution of  (\ref{*3}) from the initial condition
\begin{eqnarray}
\zeta_{0}(x,y)&=&0.4 e^{-(x^{2}/5+y^{2}/25)},\nonumber\\
v_{1}^{0}(x,y)&=&v_{2}^{0}(x,y)=0,\label{BB45}
\end{eqnarray}
while cross-sections at $y=0$ evolves as observed in Figure \ref{BBFig6}. They suggest the formation of a bi-directional $N$-waveform (cf. \cite{DMS2007}) plus small-amplitude dispersive tails behind each one.
\begin{figure}[htbp]
\centering
\subfigure
{\includegraphics[width=0.45\columnwidth]{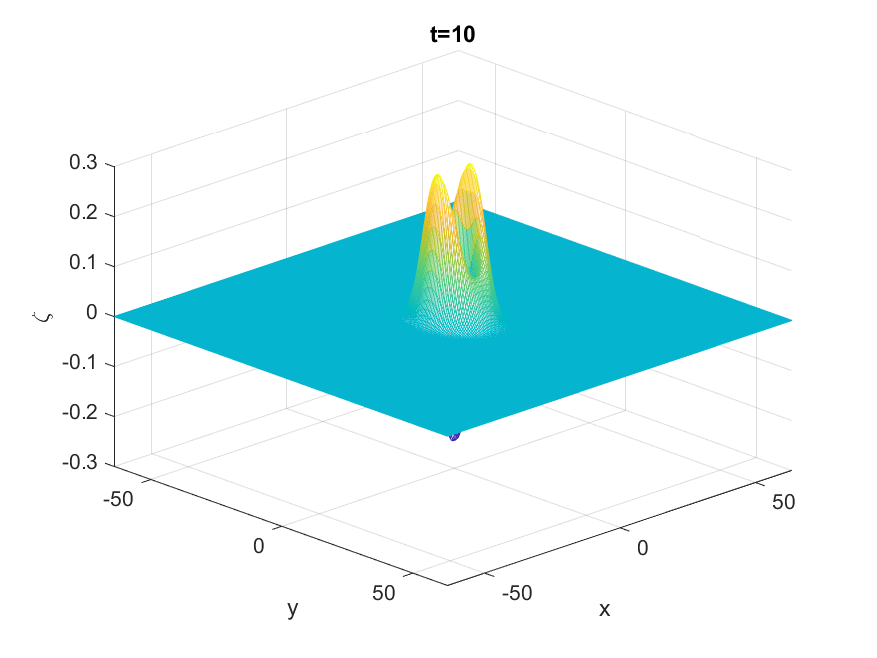}}
\subfigure
{\includegraphics[width=0.45\columnwidth]{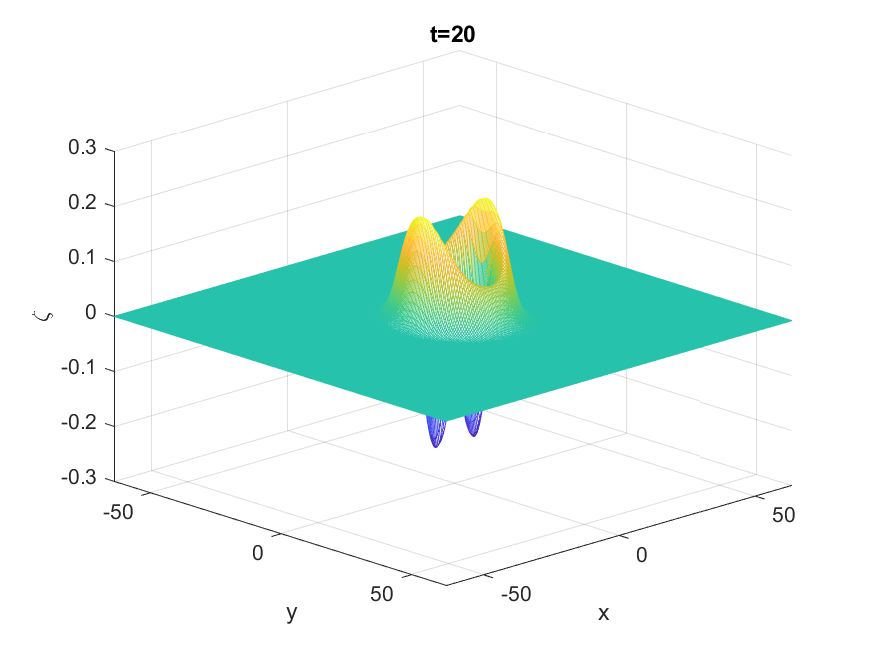}}
\subfigure
{\includegraphics[width=0.45\columnwidth]{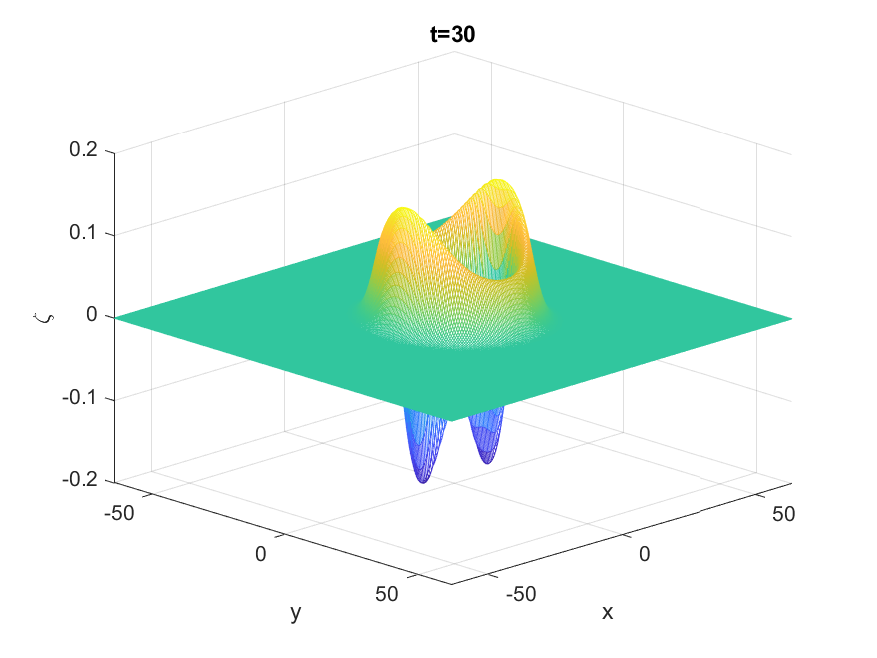}}
\subfigure
{\includegraphics[width=0.45\columnwidth]{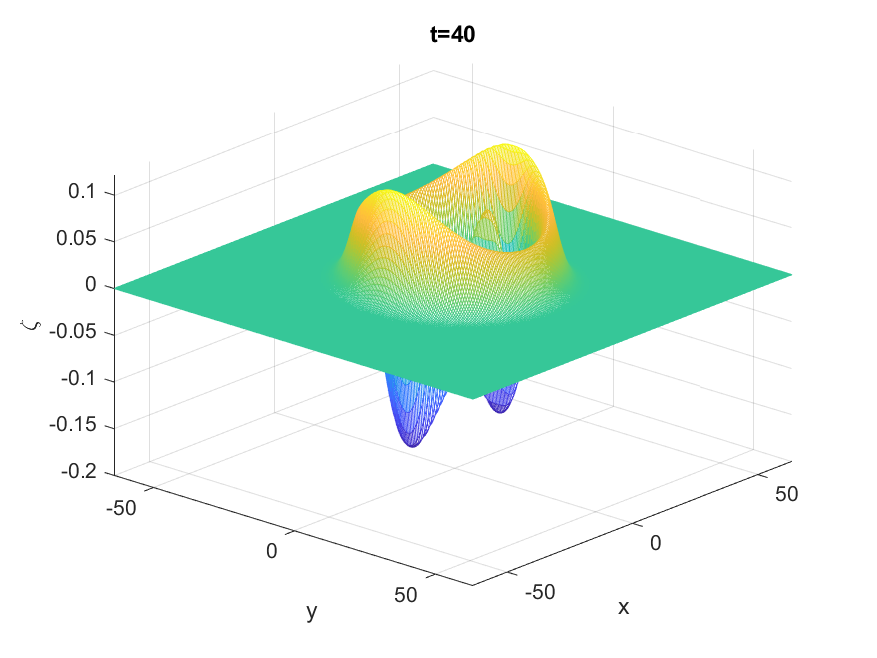}}
\caption{$\zeta$ component of the numerical solution from (\ref{BB45}) at times $t=10,20,30,40$.}
\label{BBFig5}
\end{figure}
\begin{figure}[htbp]
\centering
\subfigure
{\includegraphics[width=0.45\textwidth]{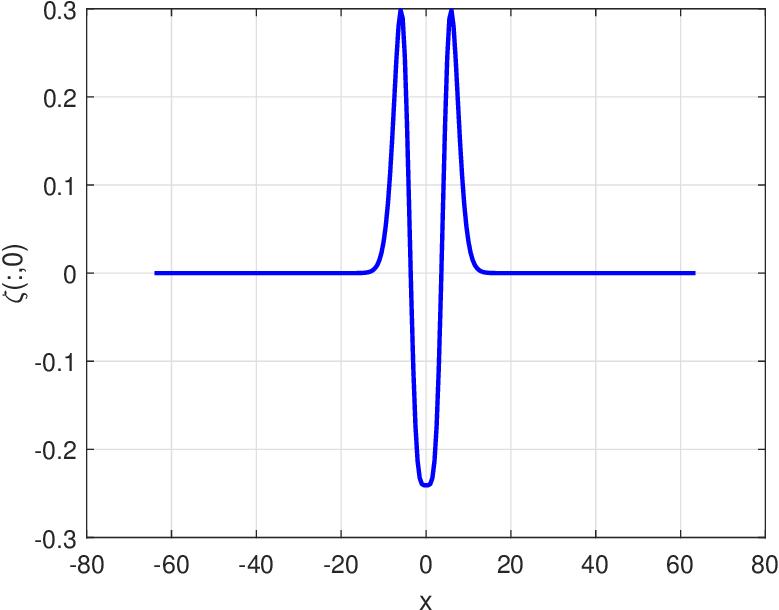}}
\subfigure
{\includegraphics[width=0.45\textwidth]{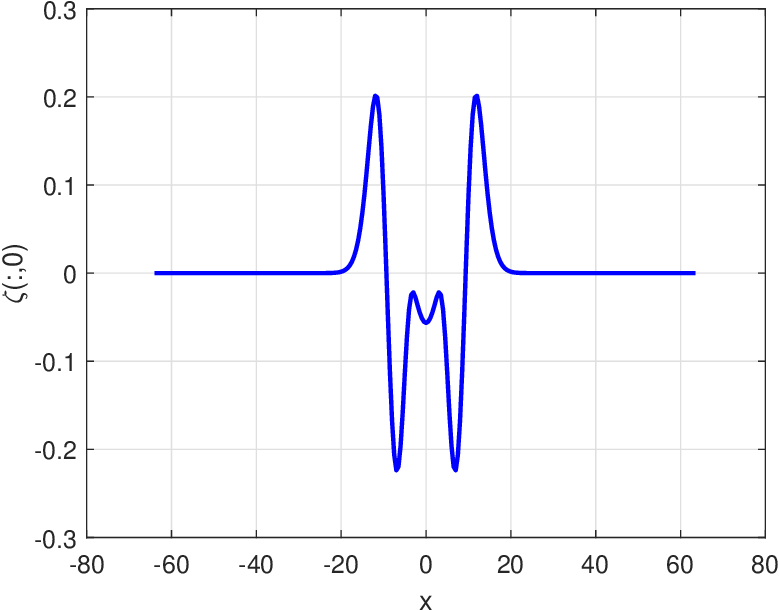}}
\subfigure
{\includegraphics[width=0.45\textwidth]{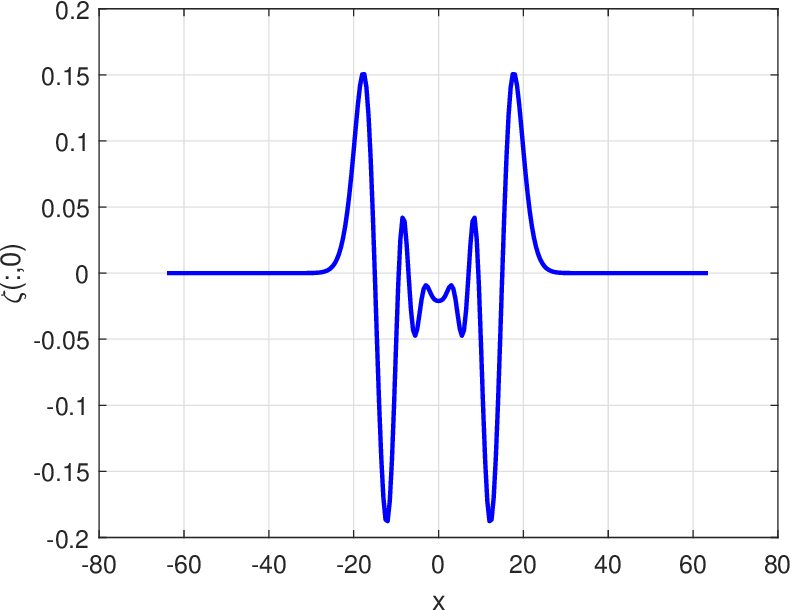}}
\subfigure
{\includegraphics[width=0.45\textwidth]{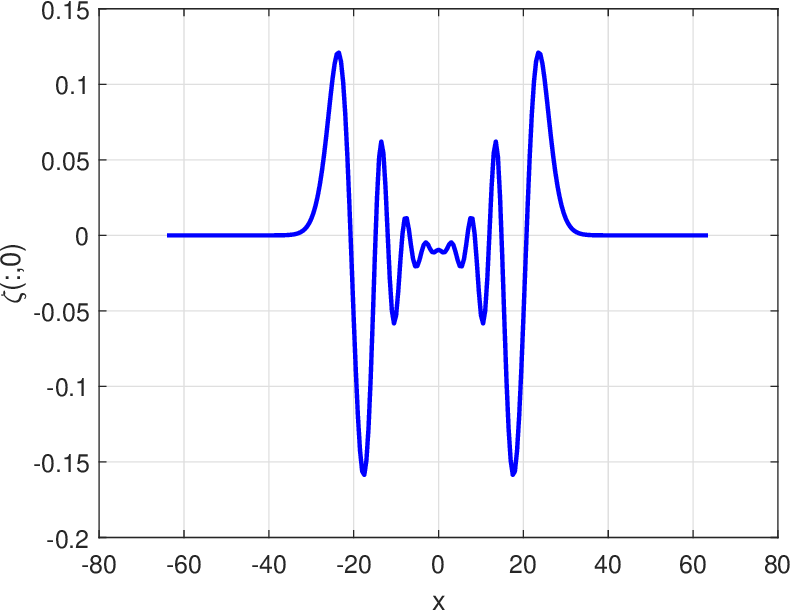}}
\caption{$\zeta$ component of the numerical solution from (\ref{BB45}) at times $t=10,20,30,40$ and $y=0$.}
\label{BBFig6}
\end{figure}
\begin{figure}[htbp]
\centering
\subfigure
{\includegraphics[width=0.45\columnwidth]{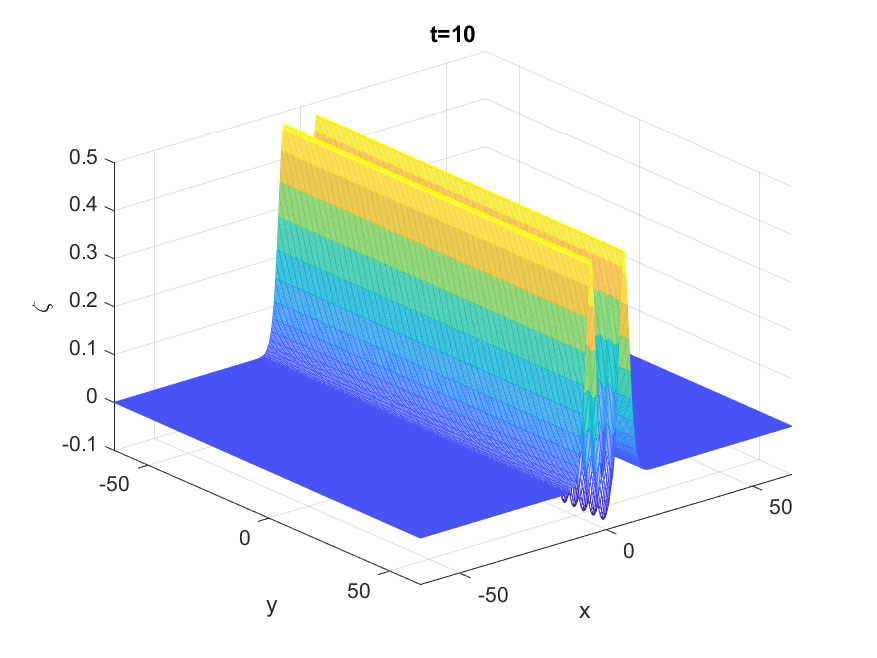}}
\subfigure
{\includegraphics[width=0.45\columnwidth]{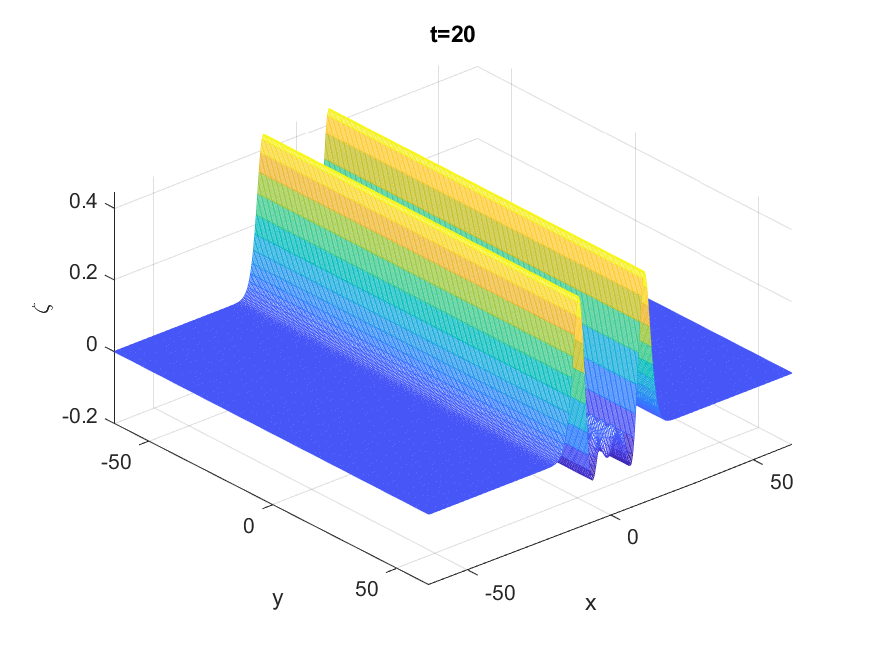}}
\subfigure
{\includegraphics[width=0.45\columnwidth]{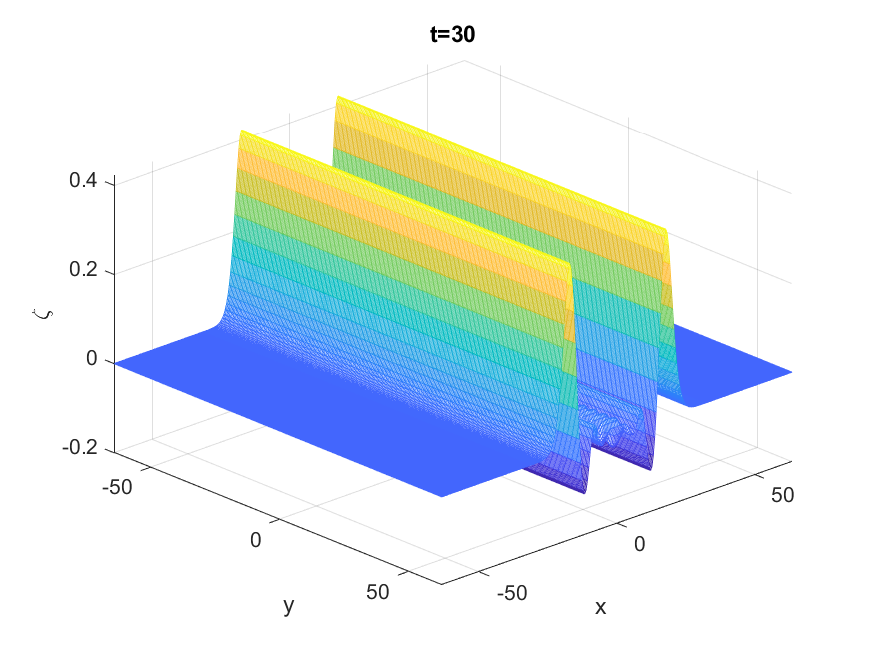}}
\subfigure
{\includegraphics[width=0.45\columnwidth]{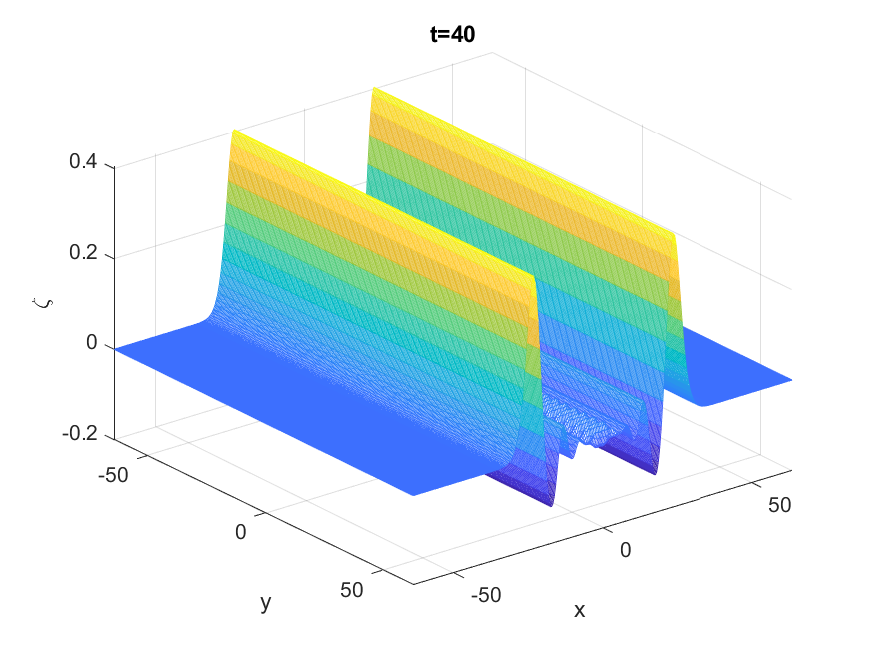}}
\caption{$\zeta$ component of the numerical solution from (\ref{BB46}) at times $t=10,20,30,40$.}
\label{BBFig7}
\end{figure}
\begin{figure}[htbp]
\centering
\subfigure
{\includegraphics[width=0.45\textwidth]{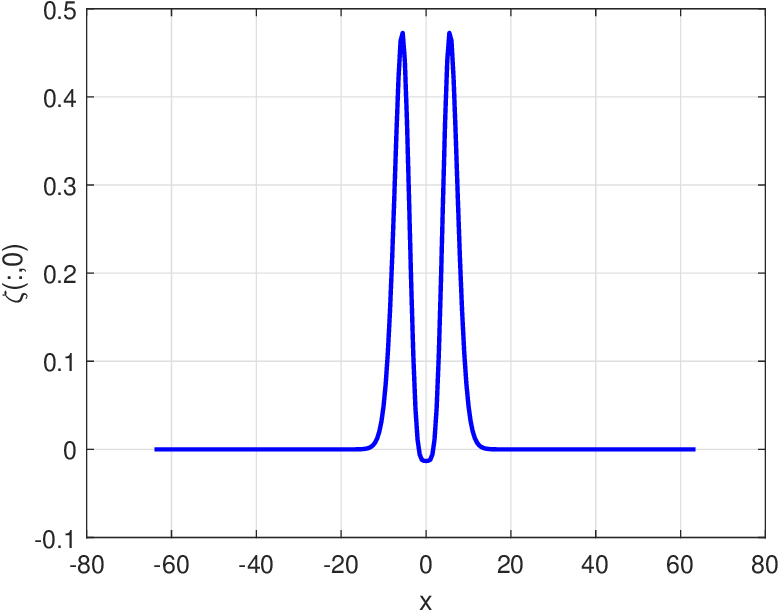}}
\subfigure
{\includegraphics[width=0.45\textwidth]{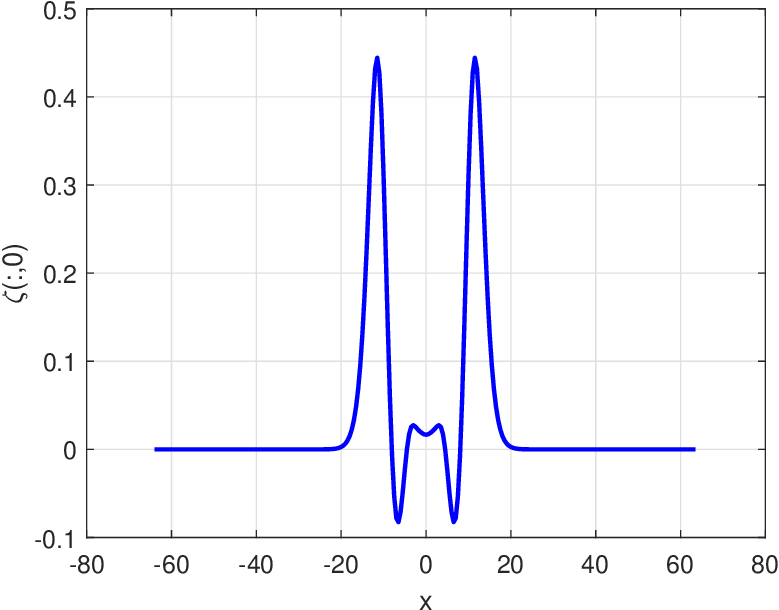}}
\subfigure
{\includegraphics[width=0.45\textwidth]{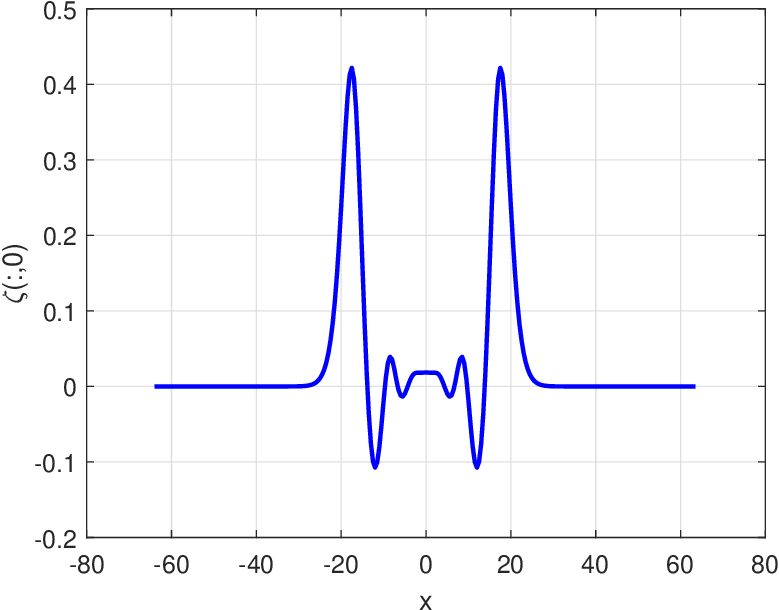}}
\subfigure
{\includegraphics[width=0.45\textwidth]{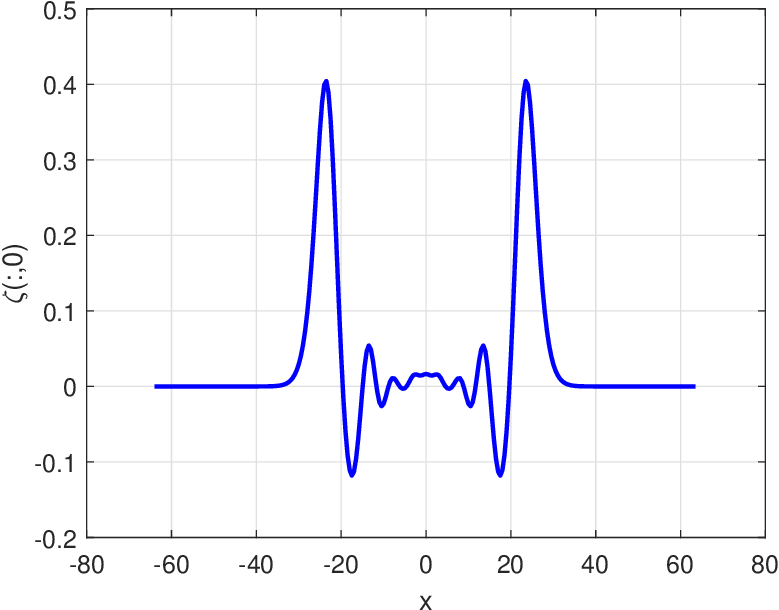}}
\caption{$\zeta$ component of the numerical solution from (\ref{BB46}) at times $t=10,20,30,40$ and $y=0$.}
\label{BBFig8}
\end{figure}

A final experiment is concerned with the generation of line solitary wave solutions. They are traveling wave solutions of the form 
$$\zeta(x,y,t)=\zeta(\xi),\quad {\bf v}(x,y,t)={\bf v}(\xi),$$ with $\xi=\alpha_{x}x+\alpha_{y}y-c_{s}t-r_{0}, r_{0}\in\mathbb{R}$, that is, traveling with permanent orm and constant speed $c_{s}$ along the direction of some ${\bf \alpha}=(\alpha_{x},\alpha_{y}), |{\bf \alpha}|^{2}=1$, and profiles $\zeta, {\bf v}$ that decay to zero as $|\xi|\rightarrow\infty$, cf. \cite{LS2023} and references therein. 

In \cite{DMS2007} a profile of the form
\begin{eqnarray}
\zeta_{0}(x,y)&=&(1+\delta\cos\frac{\pi y}{2})e^{-\frac{(x-x_{0})^{2}}{B}},\nonumber\\
v_{1}^{0}(x,y)&=&v_{2}^{0}(x,y)=0,\label{BB46}
\end{eqnarray}
was used to generate line solitary waves in the $x$-direction, with localized and decaying tranverse perturbations. Figure \ref{BBFig7} show the evolution in our case, taking $\delta=0.005, B=5, x_{0}=0$. Here the results suggest the formation of two waveforms traveling in the $x$-direction, to the right and to the left. Both seem to generate (cf. Figure \ref{BBFig8}) a line solitary wave plus small dispersive tails behind.

\end{document}